\documentclass[11pt,reqno,a4paper]{amsart}

\usepackage{amsmath}
\usepackage{enumerate}
\usepackage{amssymb}
\usepackage{amscd}
\usepackage{amsthm}
\usepackage{amsfonts}
\usepackage{graphicx}
\usepackage{hyperref}
\usepackage[matrix, arrow, curve]{xy}
\usepackage{etoolbox}

\usepackage{fouriernc}
\usepackage[T1]{fontenc}

\numberwithin{equation}{section}

\patchcmd{\subsection}{-.5em}{.5em}{}{}
\patchcmd{\subsubsection}{-.5em}{.5em}{}{}

\newcommand{\im}{\operatorname{Im}}

\newcommand{\bC}{\mathbb{C}}

\newcommand{\bH}{\mathbb{H}}

\newcommand{\bM}{\mathbb{M}}
\newcommand{\bN}{\mathbb{N}}

\newcommand{\bR}{\mathbb{R}}
\newcommand{\bS}{\mathbb{S}}
\newcommand{\bT}{\mathbb{T}}

\newcommand{\R}{\mathbb{R}}
\newcommand{\C}{\mathbb{C}}

\newcommand{\ra}{\rightarrow}

\newcommand{\qand}{\quad \textrm{and} \quad}

\newcommand\subsetsim{\mathrel{%
\ooalign{\raise0.2ex\hbox{$\subset$}\cr\hidewidth\raise-0.8ex\hbox{\scalebox{0.9}{$\sim$}}\hidewidth\cr}}}

\newcommand{\Z}{\mathbb Z}

\renewcommand{\epsilon}{\varepsilon}

\newcommand{\bd}{\underline{d}}

\theoremstyle{theorem}
\newtheorem{theorem}{Theorem}[section]
\newtheorem{corollary}[theorem]{Corollary}
\newtheorem{proposition}[theorem]{Proposition}
\newtheorem{lemma}[theorem]{Lemma}

\theoremstyle{definition}
\newtheorem{definition}[theorem]{Definition}

\newtheorem{remark}[theorem]{Remark}

\newtheorem*{example}{Example}

\renewcommand{\phi}{\varphi}

\begin{document}

\title[Aperiodic order and spherical diffraction]{Aperiodic order and spherical diffraction, III:\\ The shadow transform and the diffraction formula}
\author{Michael Bj\"orklund}
\address{Department of Mathematics, Chalmers, Gothenburg, Sweden}
\email{micbjo@chalmers.se}
\thanks{}

\author{Tobias Hartnick}
\address{Institut f\"ur Algebra und Geometrie, KIT, Karlsruhe, Germany}
\curraddr{}
\email{tobias.hartnick@kit.edu}
\thanks{}

\author{Felix Pogorzelski}
\address{Institut f\"ur Mathematik, Universit\"at Leipzig, Germany}
\curraddr{}
\email{felix.pogorzelski@math.uni-leipzig.de}
\thanks{}

\keywords{}


\date{}

\dedicatory{}

\maketitle

\begin{abstract} We define spherical diffraction measures for a wide class of weighted point sets in commutative spaces, i.e.\ proper homogeneous spaces associated with Gelfand pairs. In the case of the hyperbolic plane we can interpret the spherical diffraction measure as the Mellin transform of the auto-correlation distribution. We show that uniform regular model sets in commutative spaces have a pure point spherical diffraction measure. The atoms of this measure are located at the spherical automorphic spectrum of the underlying lattice, and the diffraction coefficients can be characterized abstractly in terms of the so-called shadow transform of the characteristic functions of the window. In the case of the Heisenberg group we can give explicit formulas for these diffraction coefficients in terms of Bessel and Laguerre functions. \end{abstract}
\section{Introduction}
This article is a culmination of a series of articles begun in \cite{BHP1} and continued in \cite{BHP2} in which we extend the diffraction theory of uniform regular model sets in abelian groups to the wide setting of proper homogeneous metric spaces. In \cite{BHP1} we introduced regular model sets in a general locally compact second countable (lcsc) group $G$, and with every such model set $\Lambda$ we associated a dynamical system $\Omega_\Lambda$ over $G$, the \emph{hull} of $\Lambda$. We then established unique ergodicity of this system and deduced that certain sampling limits over $\Lambda^{-1}\Lambda$ converge to a positive-definite Radon measure $\eta_\Lambda$ on $G$, the \emph{auto-correlation measure} of $\Lambda$. In \cite{BHP2} we generalized these results to (weighted) regular model sets in proper homogeneous metric spaces $X$. If $G$ denotes the iso\-metry group of $X$ and $K$ is one of its point stabilizers, then the auto-correlation measure of such a weighted regular model set can be seen as a positive-definite Radon measure on $K\backslash G/K$. The current article is concerned with certain Fourier transforms of these measures, which we call \emph{spherical diffraction measures} in analogy with the abelian case.

In the classical case, where $G$ is abelian and $K = \{e\}$ is the trival subgroup, the auto-correlation measure $\eta_\Lambda$ admits a Fourier transform $\widehat{\eta}_\Lambda$, which is a Radon measure on the Pontryagin dual $\widehat{G}$ of $G$. Due to its physical interpretation \cite{Dworkin-93, Hof-95}, this measure is called the \emph{diffraction measure} of $\Lambda$. A particular focus is on situations where this measure is discrete, see e.g.\@ \cite{Solomyak-98, Schlottmann-99, BaakeL-04, BaakeG-13}.
It is one of the cornerstones of the theory of quasi-crystallographic diffraction theory that if $\Lambda$ is a uniform regular model set in a locally compact abelian group, then the diffraction measure $\widehat{\eta}_\Lambda$ is pure point. More precisely, if $\Lambda$ arises from an abelian cut-and-project scheme $(G, H, \Gamma)$ with window $W \subset H$, then it follows from work of Meyer \cite{Meyer-70, Meyer-95} that
\[
\widehat{\eta}_\Lambda = \sum_{(\xi_1, \xi_2) \in \Gamma^\perp} |\widehat{{\bf 1}}_W(\xi_2)|^2 \cdot \delta_{\xi_1},
\]
where $\Gamma^\perp \subset \widehat{G} \times \widehat{H}$ denotes the dual lattice of $\Gamma$. This amounts to an exotic Poisson summation formula of the form
\[
\lim_{n \to \infty} \frac{1}{m_G(B_n)} \sum_{x \in \Lambda \cap B_n} \sum_{y \in \Lambda} f(y-x) = \sum_{(\xi_1, \xi_2) \in \Gamma^\perp} |\widehat{{\bf 1}}_W(\xi_2)|^2 \cdot \widehat{f}({\xi_1}), \quad (f \in C_c^\infty(G)),
\]
where $m_G$ denotes Haar measure of $G$ and $(B_n)$ is a suitable F\o lner sequence of balls in $G$. Our ultimate goal here is to derive similar exotic summation formulas in more general situations. While we can establish pure point diffraction in large generality, computing the diffraction coefficients explicitly will only be possible in special situations, most notably for Heisenberg groups.

One problem in generalizing Meyer's theorem beyond the abelian case is that one needs a suitable notion of Fourier transform for functions on $K\backslash G/K$. In \cite{BHP2} we considered in some details the case of the hyperbolic plane $\bH^2$. In particular, we explained how the auto-correlation measure of a weighted regular model set $\Lambda$ in $\bH^2$ can be identified with an evenly positive-definite distribution $\xi_\Lambda$ on the real line.  Due to exponential volume growth of the hyperbolic plane, the distribution $\xi_\Lambda$ is non-tempered, and hence instead of its Fourier transform one should consider its complex Fourier transform or Mellin transform. Recall that the \emph{Mellin transform} of a function $\phi \in C_c^\infty(\R)$ is  given by
\[
 \bM \phi(z) := \int_{\bR}\phi(t) e^{tz/2} \, dt \quad (z \in \C).
\]
By a theorem of Gelfand--Vilenkin and Krein \cite[Thm.\ II.6.5]{GV4}, for every evenly positive-definite distribution $\xi$ there exists a 
measure $\mu_\xi \in M^+(\C)$ with $\mathrm{supp}(\mu_\xi) \subset \R \cup i \R$ such that
\[
\xi(\phi) = \mu_\xi(\bM\phi) \quad (\phi \in C_c^\infty(\R)_{\mathrm ev}),
\]
and we call such a measure a Mellin transform of $\xi$. We are going to establish the following hyperbolic analogue of pure point diffraction in Theorem \ref{MellinPurePoint} below:
\begin{theorem}[Pure point diffraction, hyperbolic case] \label{pphyp} Assume that $\Lambda$ is a weighted uniform regular model set in $\bH^2$. Then its auto-correlation distribution $\xi_\Lambda$ has a pure point Mellin transform supported on $[-1,1] \cup i\R$.\qed
\end{theorem}
The general context in which such theorems can be established is that of spherical harmonic analysis. We will briefly explain the general formalism and then focus on the case of the Heisenberg group, in which much more precise results (in particular concerning the diffraction coefficients) can be established.

From now on let $G$ be a lcsc group, let $K<G$ be a compact subgroup  and $X = K\backslash G$. Then $(G, K)$ is called a \emph{Gelfand pair} and $X$ is called a \emph{commutative space} if the convolution subalgebra $C_c(G, K) \subset C_c(G)$ of bi-$K$-invariant functions, the so-called \emph{Hecke algebra}, is commutative. For example, the hyperbolic plane is a commutative space with $G = \mathrm{SL}_2(\R)$ and $K = \mathrm{SO}_2(\R)$. Further examples of commutative spaces include Riemannian symmetric spaces, regular trees (and more generally, Bruhat-Tits buildings) and (generalized) Heisenberg groups.

With any Gelfand pair $(G,K)$ one associated as spherical Fourier transform as follows. Denote by $\mathcal S^+(G, K)$ the set of \emph{positive-definite spherical functions}, i.e.\ matrix coefficients of irreducible unitary $G$-representations with respect to a $K$-invariant vector. Such functions are bounded, and with the restriction of the weak-$*$-topology from $L^\infty(G)$ the space $\mathcal S^+(G, K)$ is a locally compact space. (In the case of the hyperbolic plane it is homeomorphic $([-1,1] \cup i\R)/\{\pm 1\}$.) We then define the \emph{spherical Fourier transform} of $(G, K)$ by 
\[
\mathcal F: L^1(G,K) \to C_0(\mathcal S^+(G, K)),\quad f \mapsto \widehat{f}, \quad \text{where } \widehat{f}(\omega) :=  \int_G f(x) \overline{\omega(x)} \, dm_G(x).
\]
We say that a Radon measure $\widehat{\eta}$ on $\mathcal S^+(G, K)$ is a \emph{spherical Fourier transform} of a Radon measure $\eta$ on $K\backslash G/K$ if for every every $h \in {\rm span}\{f \ast g^* \mid f,g \in C_c(G,K)\}$ we have $\widehat{h} \in L^1(\mathcal S^+(G, K), \widehat{\eta})$ and 
\[ \widehat{\eta}(\widehat{h}) = 
\eta(h).
\]
Using a classical theorem of Godement \cite{Godement-57} concerning the existence and uniqueness of such spherical Fourier transforms we establish:
\begin{proposition}[Existence of spherical diffraction] If $\Lambda$ is a weighted FLC set in a commutative space $X$ and $\nu$ is an invariant measure on $\Omega_\Lambda^\times:= \Omega_{\Lambda} \setminus \{\emptyset\}$, then the corresponding autocorrelation measure $\eta_\Lambda$ admits a unique spherical Fourier transform $\widehat{\eta}_\Lambda$.  
\end{proposition}
We refer to the Radon measure $\widehat{\eta}_\Lambda$ on $\mathcal S^+(G, K)$ as the \emph{spherical diffraction} of $\Lambda$ with respect to $\nu$. Note that in the model set case there is a unique invariant measure on $\Omega_\Lambda^\times$, hence we can simply speak of the spherical diffraction of $\Lambda$. In the case of the hyperbolic plane this spherical diffraction is precisely the Mellin transform of the auto-correlation distribution.

Unlike the situation in the case of abelian groups, it is not true that the spherical diffraction of a regular model set in a commutative space is pure point. In fact, this property depends on the model set being uniform, a property which holds automatically for model sets in abelian groups (and, more generally, for approximate lattices in nilpotent groups \cite{BH}).
\begin{theorem}[Pure point spherical diffraction] \label{ppdiff} Let $\Lambda$ be a regular model set in a commutative space $X$. If $\Lambda$ is uniform (i.e.\ the underlying lattice is cocompact), then $\Lambda$ has pure point spherical diffraction.\qed
\end{theorem}
If $\Lambda$ is a uniform regular model set in $X = K\backslash G$, associated with a cut-and-project scheme $(G, H, \Gamma)$, then we can find a countable subset $C \subset \mathcal S^+(G, K)$ such that
\[
\widehat{\eta}_\Lambda = \sum_{x \in C} c(x) \cdot \delta_x.
\]
In fact, the set $C$ is simply the spherical automorphic spectrum of $\Gamma$, i.e.\ the set of matrix coefficients associated with those irreducible subrepresentations of the $G$-action on $L^2((G\times H)/\Gamma)$ which contain a $K$-invariant vector. In the classical case, where $G$ is abelian and $K$ is trivial, this set will always be a dense subset of $ \mathcal S^+(G, K) = \widehat{G}$,  but in our more general setting, new phenomena arise. For example, if $\Lambda$ is a uniform weighted regular model set in the hyperbolic plane whose underlying lattice has a strong spectral gap, then $\widehat{\eta}_\Lambda$ has an isolated atom at the constant function $1$. 

While it is easy to describe the support of the spherical diffraction measure abstractly, it is often impossible to compute it explicitly. Similarly, while we have an abstract description of the diffraction coefficients in terms of the so-called \emph{shadow transform} of the characteristic function of the window, for general commutative spaces there is no hope to compute these coefficients explicitly. A notable exception is given by Gelfand pairs $(G,K)$, for which the group $G$ is virtually nilpotent, hence we will focus on this case for the remainder of this introduction.

The easiest case beyond the abelian case considered by Meyer is that of the Euclidean motion group $G = \R^n \rtimes O(n)$ and its maximal compact subgroup $K = O(n)$. In this case, $X = K\backslash G$ is Euclidean $n$-space, and the corresponding spherical diffraction is the  ``powder diffraction'' considered already more than a decade ago in \cite{BFG}. In this case we have $\mathcal S^+(G, K) = \{\omega_\kappa \mid \kappa \geq 0\} \cong \R_{\geq 0}$, where $\omega_\kappa$ is a certain Bessel function, and if $\Lambda$ is a weighted regular model set in $X$ which arises from an irreducible $\Delta \subset \R^n \times \R^m$ and window $W \subset \bR^m$, then its diffraction is given by the formula
\[
\widehat{\eta}_\Lambda  \quad = \quad \sum_{(\sigma_1, \sigma_2) \in \Delta^\perp} |\widehat{\bf 1}_{W_o}(\sigma_2)|^2\cdot \delta_{\omega_{|\sigma_1|}}.
\]

The easiest non-virtually abelian case is that of Heisenberg motion groups, and here the diffraction formula and in particular the diffraction coefficients take already a much more involved form. To describe our results, we introduce the following notation:
\begin{itemize}
\item For $d \in \mathbb N$ we abbreviate $V_d := \C^d$ and define
\[
\beta_d: V_d \times V_d \to \R, \quad \beta_d(u,v) = -\frac{1}{2} \im \langle u,v \rangle.
\]
Then the $(2d+1)$-dimensional Heisenberg group is $N_d := \bR \oplus_{\beta_d} V_d$
\item The group $K_d := U(1)^d$ acts on $V_d$ preserving $\beta_d$, and hence acts on $N_d$ by automorphisms. The group $G_d := K_d \rtimes N_d$ is called a \emph{minimal Heisenberg motion group}. Bi-$K_d$-invariant functions on $G_d$ correspond to polyradial functions on the Heisenberg group $N_d$.
\item The space of positive-definite spherical functions decomposes into two parts as
\[
\mathcal S^+(G_d, K_d) = \{\omega_{\tau,\alpha} \mid \tau \in \R \setminus \{0\}, \alpha \in \bN^d\} \sqcup \{\omega_{0, \kappa} \mid \kappa \in \R_{\geq 0}^d\}.
\]
The horizontal part $\{\omega_{0, \kappa} \mid \kappa \in \R_{\geq 0}^d\}$ consists of products of Bessel functions in complete analogy to the virtually abelian case. 
\item The vertical part $ \{\omega_{\tau,\alpha} \mid \tau \in \R \setminus \{0\}, \alpha \in \bN^d\}$ of $\mathcal S^+(G_d, K_d)$ has no counterpart in the virtually abelian theory and is given by matrix coefficients of (infinite-dimensional) Schr\"odinger representations, which can be expressed in terms of the Laguerre polynomials $L_k$ of degree $k$ and type $0$ as given by 
\[
L_k(t) = e^{-t}\Big(\frac{d}{dt}\Big)^k(e^{t}t^k), \quad \textrm{for $k \in \bN$}.
\]
Explicitly,
\[
\omega_{\tau, \alpha}(k, t, v) = e^{i \tau t} \cdot q_{\tau, \alpha}(v), \quad \text{where}\quad q_{\tau, \alpha}(v) = e^{-|\tau||v|^2/4} \cdot \prod_{j=1}^d L_{\alpha_j}(|\tau||v_j|^2/2).
\]
\end{itemize}
We now fix $d_1, d_2 \in \bN$. We are going to construct a model set in $N_{d_1} = K_{d_1} \backslash G_{d_1}$ as follows:
\begin{itemize}
\item Let $G := G_{d_1}$ and $H := N_{d_2}$ so that
\[
G \times H = K_{d_1}\rtimes (\R^2 \oplus_\beta V_{d_1 + d_2}). 
\]
\item We choose lattices $\Delta < V_{\bd}$ and $\Xi < \R^2$ such that $\Delta$ projects densely and injectively onto $V_{d_1}$ and $V_{d_2}$, $\Xi$ projects densely and injectively onto both coordinates and such that $\beta_{\bd}(\Delta, \Delta) \subset \Xi$. We then obtain a lattice 
\[
\Gamma := \{((e, (\xi_1, \delta_1), (\xi_2, \delta_2)) \in G \times H \mid (\xi_1, \xi_2) \in \Xi, (\delta_1, \delta_2) \in \Delta\} < G \times H.
\]
For example, for $d_1 = d_2=1$ we could choose 
\[
\Delta := \{(a+b \sqrt 2 + ic + id \sqrt 2,a-b \sqrt 2 + ic - id \sqrt 2 )\mid a,b,c,d \in \Z\} < \C^2\]
and
\[
\Xi:= \{(a+b \sqrt 2,a-b \sqrt 2 )\mid a,b \in \Z\} < \R^2.
\]
For larger $d$, we could take products of such lattices or arithmetic lattices associated with higher degree number fields.
\item Given $a_j, b_j \in \R$, $0 \leq j \leq d$, we define
\[
I := [a_0, b_0] \qand W_o := \{z \in \C^d \mid  |z_j| \in [a_j, b_j]\}.
\]
Since $\Gamma$ is countable we may choose these parameters in such a way that $W := I \times W_o$ does not intersect the projection of $\Gamma$ to $H$. We then obtain a uniform regular model set
\[
\widetilde{\Lambda} := {\rm proj}_G(\Gamma \cap (G \times W)) < G,
\]
and an associated uniform regular model set $\Lambda$ in the Heisenberg group $K_{d_1} \backslash G = N_{d_1}$.
\end{itemize}
With this notation understood we derive in Theorem \ref{DiffMain} below the following explicit formula for the spherical diffraction of $\Lambda$:
\begin{theorem}[Polyradial diffraction in Heisenberg groups] \label{polyrad} The diffraction measure $\widehat{\eta}_\Lambda$ of the regular model set $\Lambda$ is given by the formula
\[
\widehat{\eta}_\Lambda  \quad = \quad \sum_{(\sigma_1, \sigma_2) \in \Delta^\perp} c_{\mathrm{hor}}(\sigma_2) \cdot \delta_{\omega_{0, |\sigma_1|}} \quad+ \quad \underset{\tau_1 \neq 0 \neq \tau_2}{\sum_{(\tau_1, \tau_2) \in \Xi^\perp}}\sum_{(\alpha, \beta) \in \bN^{d_1+d_2}} c_{\mathrm{vert}}(\alpha, \beta, \tau_1, \tau_2, \Delta) \cdot \delta_{\omega_{\tau_1, \alpha}},
\]
where the horizontal and vertical diffraction coefficients are respectively given by
\[
c_{\mathrm{hor}}(\sigma_2) \quad = \quad  |m_\R(I)|^2 \cdot |\widehat{\bf 1}_{W_o}(\sigma_2)|^2
\]
and
\[\pushQED{\qed}
c_{\mathrm{vert}}(\alpha, \beta, \tau_1, \tau_2, \Delta) \quad = \quad  \frac{|\tau_1|^{d_1}|\tau_2|^{d_2}}{(2\pi)^{d_1+d_2}}  \cdot   |\widehat{\bf 1}_I(\tau_2)|^2 \cdot |\langle {\bf 1}_{W_o}, q_{\tau_2, \beta} \rangle|^2 \cdot \sum_{(\delta_1,\delta_2) \in \Delta} q_{\tau_1,\alpha}(\delta_1) q_{\tau_2,\beta}(\delta_2).\qedhere\popQED
\]
\end{theorem}
Note that the horizontal part is in complete analogy with the virtual abelian case, whereas the vertical part (corresponding to infinite-dimensional representations) has no counterpart in the classical theory. The diffraction formula can be interpreted as an exotic Poisson summation formula for polyradial functions on the Heisenberg group in the following way: If $f \in C_c(N_{d_1})^{K_{d_1}}$ is a polyradial continuous function with compact support on $N_{d_1}$ and $B_n$ are balls in $N_{d_1}$ with respect to the Koranyi norm, then
\begin{eqnarray*} 
&&\lim_{n \to \infty}\frac{1}{m_{N_{d_1}}(B_n)} \sum_{x \in \Lambda \cap B_n} \sum_{y \in \Lambda} f(x^{-1}y)\\
&=&  \sum_{(\sigma_1, \sigma_2) \in \Delta^\perp}  |m_\R(I)|^2 |\widehat{\bf 1}_{W_o}(\sigma_2)|^2 \widehat{f}(\omega_{0, |\sigma_1|})\\
&&+\underset{\tau_1 \neq 0 \neq \tau_2}{\sum_{(\tau_1, \tau_2) \in \Xi^\perp}}\sum_{(\alpha, \beta) \in \bN^{d_1+d_2}} \sum_{(\delta_1,\delta_2) \in \Delta}  \frac{|\tau_1|^{d_1}|\tau_2|^{d_2}}{(2\pi)^{d_1+d_2}}  q_{\tau_1,\alpha}(\delta_1) q_{\tau_2,\beta}(\delta_2)  |\widehat{\bf 1}_I(\tau_2)|^2 |\langle {\bf 1}_{W_o}, q_{\tau_2, \beta} \rangle|^2 \widehat{f}({\omega_{\tau_1, \alpha}}).
\end{eqnarray*}
This article is organized as follows. In Section~\ref{sec:GelfandPairs} we recall basic facts concerning spherical harmonic analysis. In particular, we describe in Theorem~\ref{GodementConvenient}(a relative version) of the classical Godement-Plancherel theorem. An elementary proof (modulo the spherical Bochner theorem) is included in Appendix~A. In Section~\ref{sec:diffraction} this theorem is used to define spherical diffraction measures in a rather general context. Section~\ref{sec:diffractionformula} establishes pure point spherical diffraction for uniform regular (weighted) model sets as stated in Theorem~\ref{ppdiff}. We also give a general formula for the diffraction coefficients in terms of the so-called shadow transform in Theorem~\ref{DiffAbstract}. The remainder of the article is devoted to examples. In Section~\ref{SecHeisenberg} we explicitly compute the spherical diffraction  for regular model sets in Heisenberg groups as of Theorem~\ref{polyrad}, using certain estimates concerning Laguerre polynomials from Appendix~B. In Section~\ref{SecSemisimple} we explain why the spherical diffraction of regular model sets in the hyperbolic plane can be identified with the Mellin transform of the underlying auto-correlation distribution and deduce Theorem~\ref{pphyp}. 

\medskip 

\textbf{Acknowledgements.} We thank Chalmers University, G\"oteborg, Justus-Liebig-Universit\"at Gie\ss en and KIT, Karlsruhe for providing financial support as well as excellent working conditions during out mutual visits. M.B. was partially supported by L\"angmanska kulturfonden BA19-1702 and Vetenskapsr\aa det 11253320.

\section{Preliminaries on Gelfand pairs} \label{sec:GelfandPairs}
In this section we set up our notation and recall some basic results conerning Gelfand pairs. Most of the material of this subsection is fairly standard and can be found in \cite{Wolf-07, vanDijk, Dixmier, Folland-95}.

\subsection{Notational conventions}

Throughout this article, $G$ will always denote a unimodular lcsc group $G$ and $K<G$ will always denote a compact subgroup. We fix a choice of Haar measure $m_G$ on $G$ and denote by $m_K$ the Haar probability measure on $K$. We also denote by
\[
{}_Kp: G \to K\backslash G, \quad p_K: G \to G/K \qand {}_Kp_K: G \to K\backslash G/K
\]
the canonical projections. Starting from Subsection \ref{SubsecGelfandFromHere} we will always assume that $(G, K)$ is moreover a Gelfand pair (cf.\ Definition \ref{DefGelfandPair}). Our notation follows \cite{BHP2}, in particular we make the following conventions:

\begin{remark}[Notations concerning function spaces]
If $X$ is a lcsc space, then we denote by $C_c(X)$, $C_0(X)$ and $C_b(X)$ the function spaces of complex-valued compactly supported continuous functions, continuous functions vanishing at infinity and continuous bounded functions respectively.

If $(X, \nu)$ is a measure space and $f,g \in L^2(X, \nu)$, then we denote by \[\langle f, g \rangle_{X} := \langle f, g \rangle_{(X, \nu)} := \int_X f \cdot \overline{g} \, d\nu\]  the $L^2$-inner product. Following \cite{BHP2}, but contrary to the convention in \cite{BHP1}, we will choose all our inner products to be anti-linear in the second variable.

Given a function $f: G \to \C$ we denote by $\bar f$, $\check f$ and $f^*$ respectively the functions on $G$ given by \[\bar f(g) := \overline{f(g)}, \quad \check f(g) := f(g^{-1}) \quad \text{and} \quad f^*(g) := \overline{f(g^{-1})}.\]
\end{remark}

\begin{remark}[Notations concerning measures] We denote by $M(X)$ the Banach space of complex Radon measure on $X$. We write $M_b(X)$ for the subspace of finite complex measures (i.e. $\mu$ with $|\mu|(X)< \infty$), $M^+(X)$ for the subset of (positive) Radon measures and $M_b^+(X)$ for the space of bounded Radon measures on $X$. Finally we denote by ${\rm Prob}(X) \subset M^+_b(X)$ the space of probability measures on $X$. We identify $\mu \in M(X)$ with the corresponding linear functional on $C_c(X)$ and write $\mu(f) := \int_X f \, d\mu$ for $f \in C_c(X)$.
\end{remark}

The group $G$ acts on functions on $G$ by $L_gf(x) := f(g^{-1}x)$ and $R_gf(x) := f(xg)$, and dually on measures. 

\begin{remark}[Notations concerning convolution algebras] $M_b(G)$ and $L^1(G)$ are Banach-$*$-algebras under convolution. We denote by $M_b(G,K) \subset M_b(G)$ and $L^1(G, K)\subset L^1(G)$ the Banach-$*$-subalgebras consisting of measures and function classes which are bi-$K$-invariant. The spaces $M(G, K)$, $C(G,K)$, $L^p(G, K)$ etc.\ are defined similarly. The $*$-subalgebra $C_c(G, K)$ is called the \emph{Hecke algebra} and plays a central role in the current article. Averaging over $K \times K$ defines canonical retractions $M_b(G) \to M_b(G,K)$, $L^1(G) \to L^1(G, K)$, $C_c(G) \to C_c(G, K)$ etc. We denote these by $\mu \mapsto \mu^\sharp$ (in case of measures) or $f \mapsto f^\sharp$ (in case of functions).
\end{remark}

\begin{remark}[Actions of convolution algebras] If $\pi:G \to \mathcal U(V)$ is a unitary representation of $G$, then we denote by the same latter the associated $*$-representation $\pi: L^1(G) \to \mathcal B(V)$ as given by
\[
\pi(f)(u) := \int_G f(g) \pi(g)u \, dm_G(g).
\]
For the left- and right-regular representations $\pi_L, \pi_R: G \to \mathcal U(L^2(G))$, we then have \cite[Remark~A.3]{BHP2}
\begin{equation}\label{Convolution}
\pi_L(f)(u) = f \ast u \qand \pi_R(f) u = u \ast \check f.\end{equation}
\end{remark}

\begin{remark}[Canonical identifications] Pullback induces bijections ${}_Kp^*: C_c(K\backslash G) \to C_c(G)^{L(K)}$ and ${}_Kp_K^*: C_c(K\backslash G/K) \to C_c(G, K)$, and we denote their inverses by $f \mapsto {}_Kf$ and $f \mapsto {}_K f_K$ respectively. Thus for all $g \in G$, $h \in C_c(G)^{L(K)}$ and $f \in C_c(G, K)$ we have
\[
{}_K h(Kg) = h(g) \qand {}_Kf_K(KgK) = f(g).
\]
We use the same notation also for other classes of left-, respectively bi-$K$-invariant functions. The isomorphism ${}_Kp_K^*: C_c(K\backslash G/K) \to C_c(G, K)$ can be used to induce a convolution structure on $C_c(K\backslash G/K)$. For a more explicit description of this convolution structure see Definition~A.9 in~\cite{BHP2}.
\end{remark}

\begin{remark}[Convenient approximate identities]\label{ConvenientApproximateIdentity} As pointed out in \cite[Remark~A.12]{BHP2}, there exist functions $\widetilde{\rho_n} \in C_c(G)$ with the following properties:
\begin{itemize}
\item $\widetilde{\rho}_n \geq 0$, $\widetilde{\rho}_n^* = \widetilde{\rho}_n$, $\int_G \widetilde{\rho}_n \, dm_G = 1$ and all of the functions are supported inside a common pre-compact identity neighbourhood.
\item For every $1\leq p < \infty$ and $f \in L^p(G)$ we have $\widetilde{\rho}_n \ast f \to f$ and $f \ast \widetilde{\rho}_n \to f$ in $L^p$. For $f \in C_c(G)$ these convergences hold uniformly, and for $f \in C(G)$ they hold uniformly on compacta, in particular pointwise.
\item If we set $\rho_n := \widetilde{\rho}_n^\sharp$, then we have convergence ${\rho}_n \ast f \to f^\sharp$ and $f \ast {\rho}_n \to f^\sharp$ in the same sense. 
\end{itemize}
We fix such functions once and for all and refer to $(\widetilde{\rho}_n)$ and $(\rho_n)$ as convenient approximate identities in $C_c(G)$, respectively $C_c(G, K)$.
\end{remark}

\subsection{Functions and measures of positive type}
The terminology concerning positive-definite functions varies in the literature. We will use the following:
\begin{definition}\label{DefPositiveDefinite} Let $G$ be a lcsc group.
\begin{enumerate}
\item A function $\phi: G \to \C$ is called \emph{positive-definite} if for all $\lambda_1, \dots, \lambda_n \in \C$ and $x_1, \dots, x_n \in G$,
\[
\sum_{i=1}^n \sum_{j=1}^n \lambda_i \overline{\lambda_j} \phi(x_ix_j^{-1}) \geq 0.
\]
\item A function class $\phi \in L^\infty(G)$ is called \emph{of positive type} if for all $f \in L^1(G)$,
\[
\int_G (f \ast f^*)(g) \phi(g) dm_G(g) = \int_G \int_G f(g) \overline{f(h)} \phi(gh^{-1})dm_G(g)dm_G(h) \geq 0.
\]
\end{enumerate}
\end{definition}
With this terminology the following hold (\cite[Sec. 3.3]{Folland-95}): Firstly, every function class of positive type has a (unique) continuous representative, which we refer to as a function of positive type. Thus, by our convention, functions of positive type are continuous. Secondly, for \emph{continuous} functions being positive-definite and being of positive type is equivalent. More precisely:
\begin{lemma}[Characterizations of functions of positive type]\label{PosDefFct} Let $\phi \in C(G)$. Then the following are equivalent:
\begin{enumerate}[(i)]
\item $\phi$ is positive-definite.
\item $\phi$ is of positive type.
\item $\int_G (f \ast f^*)(g) \phi(g) dm_G(g) \geq 0$ for all $f \in C_c(G)$.
\item There exists a unitary representation $\pi$ of $G$ with cyclic vector $u$ such that $\phi(g) = \langle u, \pi(g) u \rangle$.
\end{enumerate}
In this case, the pair $(\pi, u)$ is unique up to isomorphism, and $\phi$ satisfies
\[
\pushQED{\qed} \|\phi\|_\infty = \|u\|^2 = \phi(e) \geq 0 \qand \phi^* = \phi.\qedhere\popQED
\]
\end{lemma}
In the sequel we denote by $P(G) \subset C(G)$ the set of continuous positive-definite functions (equivalently, functions of positive type) on $G$. We also denote by $P(G, K) := P(G) \cap C(G,K)$ the subset of bi-$K$-invariant continuous positive-definite functions. From the existence of convenient approximate identities in $C_c(G, K)$ one deduces:
\begin{lemma}\label{PDDense}
\begin{enumerate}[(i)]
\item For every $f \in C_c(G,K)$ we have $f \ast f^* \in P(G,K) \cap C_c(G,K)$.
\item The span of  $\{f \ast f^* \mid f \in C_c(G,K)\}$ is dense in $C_c(G,K)$ with respect to the topology of uniform convergence on compacta. 
\end{enumerate}
In particular, $P(G,K) \cap C_c(G, K)$ span a dense subspace of $C_c(G,K)$.
\end{lemma}
\begin{proof} (i) For all $x_1, \dots, x_n \in G$ and $\lambda_1, \dots,\lambda_n \in \C$ we have
\begin{eqnarray*}
\sum_{i,j} \lambda_i \overline{\lambda_j} (f \ast f^*)(x_ix_j^{-1}) &=& \sum_{i,j} \lambda_i \overline{\lambda_j} \int_G f(y) \overline{f(x_jx_i^{-1}y)} dm_G(y)\\
&=& \int_G  \sum_{i} \lambda_i f(x_iy) \sum_j \overline{\lambda_j} \overline{f(x_jy)} dm_G(y) \geq 0.
\end{eqnarray*}
(ii) The span contains all elements of the form $f \ast g^*$ with $f, g\in C_c(G, K)$ by polarization. Choosing a convenient approximate identity for $g$ then yields the claim.
\end{proof}
The third characterization of Lemma \ref{PosDefFct} motivates the following definition:
\begin{definition}\label{MeasureOfPositiveType}\label{MeasureOfRelPosType} A complex measure $\mu \in M(G)$ is \emph{of positive type} if $\mu(f \ast f^*) \geq 0$ for all $f \in C_c(G)$.  

A complex measure $\mu \in M(G, K)$ is \emph{of positive type relative $K$} if $\mu(f \ast f^*) \geq 0$ for all $f \in C_c(G, K)$.
\end{definition}
Note that if $\mu \in M(G)$ is of positive type in the sense of Definition \ref{MeasureOfPositiveType}, then $\mu^\sharp \in M(G, K)$ is of positive type relative to $K$ (for any choice of $K$).

\subsection{Gelfand pairs and commutative spaces}
 Under our standing assumptions that $G$ is a unimodular lcsc group and $K<G$ is a compact subgroup, the following properties of the pair $(G,K)$ are equivalent (see e.g.\ \cite[Thm. 9.8.1]{Wolf-07}); here for a unitary $G$-representation $(V, \pi_V)$ we denote by $V^K< V$ the subspace of $K$-invariant vectors.
\begin{enumerate}[(Gel1)]
\item The Hecke algebra $C_c(G, K)$ is commutative 
\item The algebra $L^1(G,K)$ is commutative.
\item The algebra $M_b(G, K)$  is commutative.
\item The $G$-representation $L^2(K\backslash G)$ is multiplicity free.
\item $\dim V^K \leq 1$ for every irreducible unitary $G$-representation $(V, \pi_V)$.
\end{enumerate}
\begin{definition}\label{DefGelfandPair}
The pair $(G,K)$ is called a \emph{Gelfand pair} if it satisfies the equivalent properties (Gel1)--(Gel5) above. In this case, the corresponding proper homogeneous space $K\backslash G$ is called a \emph{commutative space}.
\end{definition}

\subsection{Positive-definite spherical functions and the spherical Fourier transform}\label{SubsecGelfandFromHere}

From now on $(G, K)$ denotes a Gelfand pair.

\begin{proposition} Let $\omega \in C(G, K)$. Then the following are equivalent:
\begin{enumerate}[(S1)]
\item The associated Radon measure $m_\omega$ defined by
\[
m_\omega(f) := \int_G f(x) \omega(x^{-1}) \; dm_G(x) = (f\ast \omega)(e)\quad (f \in C_c(G))
\]
restricts to a character of $C_c(G, K)$, i.e.\ $m_\omega(fg) = m_\omega(f)m_\omega(g)$ for all $f,g \in C_c(G, K)$.
\item $\omega$ is not the constant $0$ function and satisfies the functional equation
\begin{equation}\label{SphericalFct}
\int_K \omega(xky) dm_K(k) = \omega(x)\omega(y) \quad (x,y \in G, k \in K).
\end{equation}
\item $\omega(e) = 1$ and $\omega$ is a joint eigenfunction for the Hecke algebra, i.e.\ for every $f \in C_c(G, K)$ there exists $\lambda_\omega(f) \in \C$ such that $f\ast \omega = \lambda_\omega(f)\omega$.
\item $\omega(e) = 1$ and for every $f \in C_c(G,K)$ we have $f \ast \omega = \widehat{f}(\omega)\cdot \omega$, where $\widehat{f}(\omega) := (f \ast \omega)(e)$.
\end{enumerate}
\end{proposition}
\begin{proof} See \cite[Prop. 6.1.5 and 6.1.6]{vanDijk} and \cite[Thm. 8.2.6]{Wolf-07}. 
\end{proof}
\begin{definition} A function $\omega \in C(G, K)$ satisfying the equivalent conditions (S1)-(S4) above is called a \emph{$(G,K)$-spherical function}. We denote by $\mathcal S(G, K)$ the set of spherical functions.
\end{definition}
If $f \in C_c(G, K)$, then we define the \emph{spherical transform} of $f$ as the function
\[
\mathbb S f: \mathcal S(G, K) \to \C, \quad \mathbb S f(\omega) :=  \int_G f(g) \omega(g^{-1}) dm_G(g) = (f \ast \omega)(e).
\]
We will consider the restrictions of this transform to the subsets  $\mathcal S^+(G, K) \subset \mathcal S_b(G, K) \subset \mathcal S(G, K)$ of positive-definite, respective bounded spherical functions. 
\begin{remark}[Topologies on $\mathcal S_b(G, K)$ and $\mathcal S^+(G, K)$]\label{TopologyS+} The space $\mathcal S_b(G, K)$ carries a natural locally compact Hausdorff topology which can be described in several ways:
\begin{enumerate}[(i)]
\item For every $\omega \in \mathcal S_b(G, K)$ the functional $m_\omega$ extends to a continuous linear functional on $L^1(G, K)$, and this defines a bijection between $\mathcal S_b(G, K)$ and the Gelfand spectrum of $L^1(G, K)$. Via this identification we obtain a locally compact topology on $\mathcal S_b(G,K)$.
\item The same topology can be described more explicitly as the restriction of the weak-$*$-topology on $L^\infty(G)$ to the subset $\mathcal S_b(G,K)$, see \cite[Sec. 6.4]{vanDijk}. Since $C_c(G) \subset L^1(G)$ is dense we thus have $\omega_n \to \omega$ in $\mathcal S_b(G, K)$ if and only if
\[
\int_G f(x)\omega_n(x)\, dm_G(x) \to \int_G f(x)\omega(x)\, dm_G(x) \quad (f \in C_c(G)).
\]
\end{enumerate} 
In the sequel we will always equip  $\mathcal S_b(G,K)$ with this locally compact topology. The subspace $\mathcal S^+(G,K) \subset \mathcal S_b(G, K)$ turns out to be closed, hence inherits a locally compact topology by \cite[Prop. 9.2.9]{Wolf-07}.
\end{remark}
Under the above identification of $\mathcal S_b(G, K)$ with the Gelfand spectrum of $L^1(G, K)$, the \emph{Gelfand transform} $\Gamma_{L^1(G, K)}: L^1(G, K) \to  C_0(\mathcal S_b(G, K))$ of the Banach algebra $L^1(G, K)$ is given as follows: For $f \in C_c(G, K)$ we have
\[
\Gamma_{L^1(G, K)}(f) = \mathbb S f|_{\mathcal S_b(G, K)},
\]
and this extends to $L^1(G, K)$. Restricting further to $\mathcal S^+(G, K)$ we obtain the following:
\begin{definition} \label{DefSphericalFT} The \emph{spherical Fourier transform} of the Gelfand pair $(G,K)$ is the transform
\[
\mathcal F: L^1(G,K) \to C_0(\mathcal S^+(G, K)),\quad f\mapsto \widehat{f} := \Gamma_{L^1(G, K)}(f)|_{\mathcal S^+(G, K)}.
\]
\end{definition}
Since by Lemma \ref{PosDefFct} any $\omega \in \mathcal S^+(G, K)$ satisfies $\omega = \omega^*$, we have the explicit formula
\begin{equation}\label{SFT}
\widehat{f}(\omega) : =   \int_G f(x) \overline{\omega(x)} \, dm_G(x) = (f \ast \omega)(e) = \langle f, \omega \rangle \quad (f \in  L^1(G,K), \omega \in \mathcal S^+(G, K))).
\end{equation}

We record for later use the formula
\begin{equation}\label{FTfbar}
\widehat{\overline{f}}(\omega) = \overline{\int_G f(x)\overline{\check \omega(x)}}\, dm_G(x) = \overline{\widehat{f}(\check \omega)}
\end{equation}
If $G$ is abelian and $K = \{e\}$, then the positive-definite spherical functions are precisely the characters of $G$ and hence the spherical Fourier transform of $(G, \{e\})$ coincides with the classical Fourier transform of $G$. In this case we have for all $\omega \in \mathcal S^+(G, \{e\})$ the formula  $\widehat{L_xf}(\omega) = \widehat{f}(\omega) \cdot \omega(x)$ and $\widehat{R_xf}(\omega) = \widehat{f}(\omega) \cdot \overline{\omega(x)}$. This generalizes as follows:
\begin{lemma}[Spherical Fourier transform and translations]\label{LeftTranslationSharp} Let $f \in C_c(G, K)$ and denote $L_x^\sharp f := L_x(f)^\sharp \in C_c(G, K)$ and $R_x^\sharp f := R_x(f)^\sharp \in C_c(G, K)$. Then
\begin{equation}\label{FTTranslation}
\widehat{L_x^\sharp f}(\omega)  \widehat{f}(\omega) \cdot {\omega(x)} \qand \widehat{R_x^\sharp f}(\omega) = \widehat{f}(\omega) \cdot \overline{\omega(x)}.
\end{equation}
\end{lemma}
\begin{proof}
By the functional equation (S2) we have
\begin{eqnarray*}
\widehat{L_x^\sharp f}(\omega) &=& \int_G \int_K f(x^{-1}ky) \omega(y^{-1})\, dm_K(k)\, dm_G(y)
\quad = \quad \int_G \int_K f(y) \omega(y^{-1}x^{-1}k)\, dm_K(k)\, dm_G(y)\\
&=& \int_G f(y) \omega(y^{-1}x)\, dm_G(y) \quad = \quad  \int_G  f(k^{-1}y) \int_K\omega(y^{-1}kx)\, dm_K(k)\, dm_G(y) \\
&=& \int_G f(y) \omega(y^{-1}) \omega(x) dm_G(y) \quad = \quad \widehat{f}(\omega) \cdot {\omega(x)}.
\end{eqnarray*}
The computation for the right-translation action is similar.
\end{proof}
\subsection{Spherical representations and matrix coefficients}
\begin{definition} A unitary representation $\pi_W: G \to \mathcal U(W)$ of $G$ is called \emph{$K$-spherical} if the subspace $W^K$ of $K$-invariants is non-trivial.
\end{definition}
If $(V, \pi_V)$ an irreducible unitary $G$-representation, then by characterization (Gel5) of a Gelfand pair $(V, \pi_V)$ is $K$-spherical if and only if $\dim V^K = 1$. In this case the matrix coefficient
\begin{equation}\label{omegaV}
\omega_V: G \to \C, \quad \omega_V(g):= \langle v, \pi_V(g). v\rangle
\end{equation}
is independent of the unit vector $v \in V^K$ used to define it, and we refer to $\omega_V$ simply as the \emph{spherical matrix coefficient of $V$}.
According to  \cite[Thm.\ 8.4.8]{Wolf-07} the assignment $(V, \pi_V) \mapsto \omega_V$ induces a bijection between the set of unitary equivalence classes of irreducible spherical representations and the set $\mathcal S^+(G, K)$ of all positive-definite spherical functions. 

If $(W, \pi_W)$ is a spherical representation and $\omega \in \mathcal S^+(G, K)$, then we denote by $W_\omega$ the $(V, \pi_V)$-isotypical component of $W$, where $(V, \pi_V)$ is an irreducible spherical representation with $\omega_V = \omega$. By definition, 
$W_\omega$ is the unique maximal subspace of $W$ which is isomorphic to a direct sum of copies of $(V, \pi_V)$.

Recall that if $(W, \pi_W)$ is any unitary representation of $G$, then it induces a $*$-representation (denoted by the same letter)
\[
\pi_W: L^1(G) \to \mathcal B(W), \quad \pi_W(f)(w) = \int_G f(g) \pi_W(g)w\, dm_G(g),
\]
and $\pi_W(L^1(G))$ preserves irreducible subspaces of $W$. Moreover, the subalgebra $\pi_W(L^1(G, K))$ maps $W$ onto the subspace $W^K$, and hence preserves the latter. The action of the subalgebra $\pi_W(C_c(G,K))$ on this subspace is given by the spherical Fourier transform in the following sense; here given $\omega \in \mathcal S^+(G, K)$ we denote $W_\omega^K := W_\omega \cap W^K$.
\begin{lemma}[Action of the Hecke algebra on spherical representations]\label{ActionByFourierCoefficient} Let $\omega \in \mathcal S^+(G,K)$, $f \in C_c(G, K)$ and $u \in W_\omega^K$. Then \[\pi_W(f) u = \widehat{f}(\omega) u.\]
\end{lemma}
\begin{proof} We start with some easy reductions: Firstly, we may assume that $u$ is a unit vector. Secondly, it suffices to prove the claim 
in the case where $W = W_\omega \cong \bigoplus_I (V, \pi_V)$. Finally, by considering the various components of $u$ in this decomposition separately, one may assume that $W = V$ is irreducible. We then have $\dim_\C W^K = 1$, and since $\pi_W(f)u$ is $K$-invariant there exists $\lambda \in \C$ such that $\pi_W(f) u = \lambda \cdot u$. To determine $\lambda$ we compute
\begin{eqnarray*}
\lambda &=& \lambda \langle u, u \rangle \quad = \quad \langle \pi_W(f) u , u \rangle = \left \langle  \int_G f(g) \pi_W(g)u\, dm_G(g), u \right \rangle\\
&=& \int_G f(g) \langle \pi_W(g)u, u \rangle dm_G(g) \quad =\quad \int_G f(g) \langle u, \pi_W(g^{-1})u \rangle dm_G(g) \\
&=& \int_G f(g) \omega(g^{-1}) dm_G(g) \quad = \quad \widehat{f}(\omega).
\end{eqnarray*}
This finishes the proof.
\end{proof}


\subsection{The Godement-Plancherel theorem}
The goal of this subsection is to explain how to extend the spherical Fourier transform to certain classes of Radon measures.
\begin{proposition}\label{GConditions} Let $\mu \in M(G)$ be a complex measure. Then for a complex measure $\widehat{\mu}$ on $\mathcal S^+(G, K)$ the following three equivalent conditions are equivalent:
\begin{enumerate}[(God1)]
\item For every $h \in {\rm span}\{f \ast g^* \mid f,g \in C_c(G,K)\}$ we have $\widehat{h} \in L^1(\mathcal S^+(G, K), \widehat{\mu})$ and 
\[
\mu(h) = \widehat{\mu}(\widehat{h}).
\]
\item For every $f \in C_c(G,K)$ we have $\widehat{f}\in L^2(\mathcal S^+(G, K), \widehat{\mu})$ and 
\[
\mu(f \ast f^*) =  \|\widehat{f}\|_{L^2(\mathcal S^+(G, K), \widehat{\mu})}^2= \widehat{\mu}(|\widehat{f}|^2).
\]
\item For all $f,g \in C_c(G, K)$ we have $\widehat{f}, \widehat{g}\in L^2(\mathcal S^+(G, K), \widehat{\mu})$ and 
\[
\mu(f \ast g^*) = \langle \widehat{f}, \widehat{g} \rangle_{L^2(\mathcal S^+(G, K), \widehat{\mu})} = \int_{\mathcal S^+(G, K)} \widehat{f} \overline{\widehat{g}} \; d\widehat{\mu}.
\]
\end{enumerate}
\end{proposition}
\begin{proof} (God1) applied to $f\ast f^*$ yields (God2),  (God2) implies (God3) by the polarization identity, and (God3) implies (God1) by plugging in a convenient approximate identity as in Remark \ref{ConvenientApproximateIdentity} for $g$.
\end{proof}
\begin{definition}\label{DefAssociatedMeasure} A measure $\widehat{\mu}$ satisfying the equivalent conditions of Proposition \ref{GConditions} is called a \emph{spherical Fourier transform} of $\mu$.
\end{definition}
Note that, by definition, the spherical Fourier transform (if it exists) only depends on the restriction $\mu|_{C_c(G, K)}$ of $\mu$ to bi-$K$-invariant functions. We now discuss the existence and uniqueness of spherical Fourier transforms. The following is the most general statement that we will need in the current article; for this we recall from Definition \ref{MeasureOfRelPosType} the notion of a measure of positive type relative $K$.
\begin{theorem}[Godement-Plancherel, relative version]\label{GodementConvenient}
If $\mu \in M(G, K)$ is of positive type relative $K$, then $\mu$ has a unique spherical Fourier transform $\widehat{\mu}$, which is a positive Radon measure. Moreover, $\mu$ is uniquely determined by $\widehat{\mu}$.
\end{theorem} 
This is a slight generalization of a theorem from  \cite{Godement-57}. The original version is as follows:
\begin{corollary}[Godement-Plancherel, absolute version]\label{GodementPlancherel} If $\mu \in M(G)$ is of positive type, then $\mu$ has a unique spherical Fourier transform $\widehat{\mu}$, which is a positive Radon measure.
\end{corollary}
\begin{proof} If $\mu$ is of positive type, then $\mu^\sharp \in M(G, K)$ is of positive type relative $K$, and we have
\[
\mu|_{C_c(G, K)} = \mu^\sharp|_{C_c(G, K)}.
\]
Since the spherical Fourier transform of a measure only depends on its restriction to $C_c(G, K)$, we have thus reduced to the relative case.
\end{proof}
We explain how Theorem \ref{GodementConvenient} can be deduced from the more classical spherical Bochner theorem in Appendix \ref{AppendixGodement}. For a detailed account of Godement's original proof see \cite[Chapter XV, Sec.\ 9]{Dieudonne}. As the name indicates, Corollary \ref{GodementPlancherel} implies the classical Plancherel theorem for the Gelfand pair $(G, K)$:
\begin{example} The measure $\mu = \delta_e \in M(G)$ is of positive type, and its Fourier transform $\nu_{(G, K)} := \widehat{\delta_e}$ is called the \emph{Plancherel measure} of the Gelfand pair $(G, K)$. By (G2) we have $\widehat{f} \in L^2(\mathcal S^+(G, K), \nu)$ for all $f\in C_c(G)$, and
\begin{equation}\label{PlancherelIsometry}
\|f\|^2_2 = \delta_e(f\ast f^*) = \|\widehat{f}\|^2_{L^2(\mathcal S^+(G, K), \nu_{(G, K)}}.
\end{equation}
\end{example}
\begin{corollary}[Spherical Plancherel theorem]\label{SphericalPlancherel} If $\nu$ denotes the Plancherel measure of the Gelfand pair $(G, K)$, then the map $\mathcal F: L^1(G, K) \cap L^2(G, K) \to L^2(\mathcal S^+(G, K), \nu_{(G, K)})$, $f \mapsto \widehat{f}$ extends continuously to an isometry
\[
\mathcal F_{L^2}: L^2(G, K) \to L^2(\mathcal S^+(G, K), \nu_{(G, K)}).
\]
\end{corollary}
\begin{proof} It is immediate from \eqref{PlancherelIsometry} that $\mathcal F$ extends to an isometric embedding $\mathcal F_{L^2}$. Since $\mathcal F(C_c(G))$ is dense in $C_0(\mathcal S^+(G, K))$, it is in particular dense in $C_c(\mathcal S^+(G, K))$ and $L^2(\mathcal S^+(G, K), \nu_{(G, K)})$, hence the corollary follows. 
\end{proof}
Note that if $H < G$ is a closed subgroup, then Corollary \ref{GodementPlancherel} can also be applied to the Haar measure $m_H$ of $H$. This yields a spherical Plancherel theorem for the homogeneous space $G/H$.

\section{Spherical Diffraction} \label{sec:diffraction}
\subsection{General setting}
Throughout this section, $(G, K)$ denotes a Gelfand pair. From now on we reserve the letter $X$ to denote the associated commutative space $X = K\backslash G$, on which $G$ acts by  $g.(Kh) := Khg^{-1}$. There is a canonical measure $m_X$ on $X$ such that
\[
\int_G f \, dm_G = \int_X \left(\int_K f(kg)\, dm_K(k)\right) dm_X(Kg) \quad (f \in C_c(G)),
\]
and we denote by ${}_K\pi_R: G \to \mathcal U(L^2(X, m_X))$ the corresponding unitary representation.

In \cite{BHP2} we have introduced the notion of a translation-bounded measure $\mu$ on $X$. Typical examples of such measures are given by weighted model sets, i.e. measures of the form ${}_K p_*\delta_\Lambda$, where $\Lambda$ is a regular model set in $G$, $\delta_\Lambda$ is the associated Dirac comb and ${}_Kp: G \to K\backslash G$ is the canonical projection.

Throughout this section we will assume that $\mu$ is a translation bounded measure on $X$ satisfying the following assumptions:
\begin{enumerate}[(H1)]
\item The punctured hull $\Omega_\mu^\times:= \Omega_{\mu} \setminus \{\emptyset \}$ is uniformly locally bounded, cf.\@ Section~3.3 in \cite{BHP2}.
\item There exists a $G$-invariant probability measure on $\Omega_\mu^\times$.
\end{enumerate}
Both assumptions are automatically satisfied in the case of weighted model sets. We then fix a $G$-invariant measure $\nu$ on $\Omega_\mu^\times$. Everything in the sequel will depend on this choice of measure. Note however that in the case of weighted model sets the invariant measure is unique.
\begin{remark}[Notation concerning the Koopman representation] We will denote by $\pi_\nu$ the unitary representation
\[
\pi_\nu: G \to \mathcal U(L^2(\Omega_\mu^\times, \nu)), \quad \pi_\nu(g)u(\mu') := u(g^{-1}_*\mu'),
\]
as well as the associated $*$-representation given by
\[
\pi_\nu: L^1(G) \to \mathcal B(L^2(\Omega_\mu^\times, \nu)), \quad \pi_\nu(f)(u)(\mu') = \int_G f(g) u(g^{-1}_*\mu') dm_G(g). 
\]
If $(V, \pi_V)$ is an irreducible spherical representation with spherical matrix coefficient $\omega = \omega_V \in \mathcal S^+(G, K)$, then we denote by \[
{\rm proj}_\omega:L^2(\Omega_\mu^\times, \nu) \to L^2(\Omega_\mu^\times, \nu)_\omega
\]
the projection onto the corresponding isotypical component and set
\[L^2(\Omega_\mu^\times, \nu)^K_\omega := L^2(\Omega_\mu^\times, \nu)^K \cap L^2(\Omega_\mu^\times, \nu)_\omega.\]
\end{remark}
\subsection{The periodization map}
In \cite{BHP2} we defined a  \emph{periodization map} $\mathcal P_\mu: C_c(K\backslash G/K) \to C_0(\Omega_\mu^\times)$, and our standing assumption (H1) implies that this map is actually continuous. With out current notation we have
\[
\mathcal P_\mu({}_K f_K)(\mu') = \mu'({}_Kf) =\int_{K\backslash G} {}_Kf(x) d\mu'(x) \quad(f \in C_c(G, K), \mu' \in \Omega_\mu^\times).
\]

It is immediate from this explicit formula that $\mathcal P_\mu$ takes values in the subspace $C_0(\Omega^\times_\mu)^K$ of $K$-invariant functions. 
The goal of this subsection is to establish the following projection formula:
 \begin{theorem}[Projection formula]\label{ProjectionFormula}
 For every $\omega \in \mathcal S^+(G, K)$ there exists a constant $c_\nu(\omega) \geq 0$ such that for all $f \in C_c(G, K)$,
 \[
  \|{\rm proj}_\omega(\mathcal P_\mu ({}_Kf_K))\|^2_{L^2(\Omega_\mu^\times, \nu)}  = c_\nu(\omega) \cdot |\widehat{f}(\check \omega)|^2
 \] 
 \end{theorem}
\begin{definition} The constants $c_\nu(\omega)$ are called the \emph{diffraction coefficients} of $\nu$.
 \end{definition}
The reason for this terminology will become apparent in Theorem \ref{AbstractPurePoint} below.\\

The proof of the projection formula is based on the following lemma: 
\begin{lemma}\label{ProjectionFormulaLemma1}
For $\rho,f \in C_c(G, K)$ we have $\mathcal P_\mu ({}_K(\rho \ast f)_K) = \pi_\nu(\check f)(\mathcal P_\mu({}_K\rho_K)) $.
\end{lemma}
\begin{proof} For $g \in G$ and $\mu' \in \Omega_\mu^\times$ we have
\[
\pi_\nu(g)\mathcal P_\mu({}_K\rho_K)(\mu') = \mathcal P_\mu({}_K\rho_K)(g^{-1}_*\mu') = \int_{K\backslash G} {}_K\rho(xg)\,d\mu'(x).
\]
We thus obtain
\begin{eqnarray*}
\pi_\nu(\check f)(\mathcal P_\mu({}_K\rho_K)(\mu')) &=& \int_G \check f(g) \int_{K\backslash G} {}_K\rho(xg)\,d\mu'(x)\,dm_G(g) \quad=\quad  \int_{K\backslash G} {}_K\pi_R(\check f)({}_K\rho)(x) d\mu'(x)\\
&=& \int_G {}_K(\pi_R(\check f)(\rho))(x) d\mu'(x) \quad = \quad \mathcal P({}_K(\pi_R(\check f)(\rho))_K)(\mu').
\end{eqnarray*}
Since $\pi_R(\check f)(\rho)= \rho \ast f$, the lemma follows.
\end{proof}
We will apply this as follows:
\begin{corollary}\label{ProjectionFormulaCor1}  For $\rho, f \in C_c(G, K)$ and $u \in L^2(\Omega_\mu^\times, \nu)^K_\omega$ we have
\[
\langle \mathcal P_\mu({}_K(\rho \ast f)_K), u\rangle = {\widehat{f}(\check \omega)} \cdot \langle \mathcal P_\mu({}_K\rho_K), u\rangle.
\]
\end{corollary}
\begin{proof} Since $\pi_\nu$ is a $*$-representation, Lemma \ref{ProjectionFormulaLemma1} yields
\[
\langle \mathcal P_\mu {}_K(\rho \ast f)_K, u\rangle = \langle \pi_\nu(\check f)(\mathcal P_\mu({}_K\rho_K))  ,u\rangle =\langle \mathcal P_\mu({}_K\rho_K),  \pi_\nu(\check f^*)(u) \rangle = \langle \mathcal P_\mu({}_K\rho_K),  \pi_\nu(\overline{f})(u) \rangle.
\]
By Lemma \ref{ActionByFourierCoefficient} and \eqref{FTfbar} we have
\[
\pi_\nu(\overline{f})u = \widehat{\overline{f}}(\omega)\cdot u = \overline{\widehat{f}(\check \omega)} \cdot u.
\]
The corollary follows.
\end{proof}
We also need to use the properties of our convenient approximate identity $(\rho_n)$ as discussed in Remark \ref{ConvenientApproximateIdentity} in the following form:
\begin{lemma}
\label{UniformApproximation}
For every $f \in C_c(G, K)$, the sequence $(\rho_n \ast f)$ converges uniformly to $f$ in $C_c(G, K)$ and the sequence $(\mathcal P_\mu({}_K({\rho}_n\ast f)_K)$ converges uniformly to $\mathcal P_\mu({}_Kf_K)$ in $C_0(\Omega^\times_\mu)$.
\end{lemma}
\begin{proof} Let $U \subset G$ be a pre-compact set such that $f$ and all of the functions $f \ast \rho_n$ are supported inside $U$. We then have the estimate
\[
\left|\mathcal P_\mu({}_K(f \ast {\rho}_n)_K) - \mathcal P_\mu({}_Kf_K)\right| = \left|\int_{K\backslash G} {}_K(f \ast \rho_n - f)(x) d\mu'(x)\right|  \leq \mu'({}_Kp(U)) \cdot \|f \ast \rho_n - f\|_\infty.
\]
Now the first factor is bounded, since $\Omega_\mu$ is uniformly locally bounded. Since $f \ast \rho_n \to f$ uniformly on $U$, the lemma follows.
\end{proof}
 \begin{proof}[Proof of Theorem \ref{ProjectionFormula}] Let $(u_\alpha)_{\alpha \in I_\omega}$ be an orthonormal basis of $L^2(\Omega_\mu^\times, \nu)^K_\omega$. Since $\mathcal P_\mu({}_K f_K)$ is $K$-invariant, so is its projection onto $L^2(\Omega_\mu^\times, \nu)_\omega$, and thus 
\[
\|{\rm proj}_\omega(\mathcal P_\mu({}_Kf_K)\|^2_{L^2(\Omega_\mu^\times, \nu)} = \sum_{\alpha \in I_\omega} |\langle \mathcal P_\mu ({}_Kf_K), u_\alpha\rangle|^2 
\]
By Lemma \ref{UniformApproximation} we have uniform convergence $\mathcal P_\mu({}_K({\rho}_n\ast f)_K) \to \mathcal P_\mu({}_Kf_K)$, which implies convergence in $L^2$. We deduce that
\[
\|{\rm proj}_\omega(\mathcal P_\mu({}_Kf_K)\|^2_{L^2(\Omega_\mu^\times, \nu)} = \sum_{\alpha \in I_\omega} \lim_{n \to \infty} |\langle\mathcal P_\mu({}_K({\rho}_n\ast f)_K, u_\alpha \rangle|^2.
\]
By Corollary \ref{ProjectionFormulaCor1} we have for every $\alpha \in I_\omega$
\[
\langle\mathcal P_\mu({}_K({\rho}_n\ast f)_K, u_\alpha \rangle = {\widehat{f}(\check \omega)} \cdot \langle \mathcal P_\mu({}_K\rho_K), u_\alpha \rangle,
\]
hence
\[
\|{\rm proj}_\omega(\mathcal P_\mu({}_Kf_K)\|^2_{L^2(\Omega_\mu^\times, \nu)} = |\widehat{f}(\check \omega)|^2 \cdot  \sum_{\alpha \in I_\omega} \lim_{n \to \infty}  |\langle \mathcal P_\mu({}_K\rho_K), u_\alpha \rangle|^2.
\]
Since the second factor is independent of $f$, the theorem follows.
\end{proof}

The proof of Theorem \ref{ProjectionFormula} yields the formula
\[
c_\nu(\omega) ={\sum_{\alpha \in I_\omega} \lim_{n \to \infty} |\langle \mathcal P_\mu \rho_n, u_\alpha \rangle|^2},
\]
for the diffraction coefficients, but this formula is hard to evaluate in praxis. We will later find more explicit formulas in special cases.

\subsection{Auto-correlation measure and spherical diffraction}
In \cite{BHP2} we defined the notion of an auto-correlation measure $\eta \in M^+(K\backslash G/K)$ associated with the invariant measure $\nu$ on $\Omega_\mu^\times$. Our standing assumption (H1) ensures that this measure is well-defined, and it is uniquely determined by the fact that for all $f \in C_c(K\backslash G/K)$ we have
\begin{equation}\label{Autocorrelation}
\eta(f \ast f^*) = \|\mathcal P_\mu (f)\|^2_{L^2(\Omega_\mu^\times, \nu)}.
\end{equation}
In fact, to obtain a formula for $\eta(f)$ we can polarize \eqref{Autocorrelation}: Let $\rho \in C_c(G, K)$ be our convenient identity an define $\rho^\dagger_n := {}_K(\rho_n)_K \in C_c(K\backslash G/K)$. Then $f \ast (\rho^\dagger_n)^* = f \ast \rho_n^\dagger$ converges uniformly to $f$, and hence
\[
\eta(f) = \lim_{n \to \infty} \eta(f \ast ({\rho}^\dagger_n)^*) = \lim_{n \to \infty} \langle \mathcal P_\mu (f), \mathcal P_\mu ({\rho}^\dagger_n)\rangle_{L^2(\Omega_\mu^\times, \nu)}.
\]
\begin{remark}[Construction of the diffraction measure]
The auto-correlation measure $\eta$ corresponds via the isomorphism $M^+(G, K) \cong M^+(K\backslash G/K)$ to a bi-$K$-invariant Radon measure $\widetilde{\eta}$ on $G$. Explicitly, if $f \in C_c(G)$, then
\[
\widetilde{\eta}(f) = \eta({}_K{f^\sharp}_K),
\] 
Note that for every $f \in C_c(G, K)$ we have
\[
\widetilde{\eta}(f \ast f^*) =  \|\mathcal P_\mu( {}_K f_K)\|^2_{L^2(\Omega_\mu^\times, \nu)} \geq 0,
\]
i.e.\ $\widetilde{\eta} \in M^+(G, K)$ is of positive type relative $K$. By the Godement--Plancherel theorem (Theorem \ref{GodementConvenient}) it thus admits a Fourier transform, which is a positive Radon measure on $\mathcal S^+(G, K)$. We denote this Fourier transform by $\widehat{\eta}$, and observe that by Theorem \ref{GodementConvenient}) $\widetilde{\eta}$ and consequently $\eta$ are uniquely determined by $\widehat{\eta}$.
\end{remark}
\begin{definition} The measure $\widehat{\eta} \in M^+(\mathcal S^+(G, K))$ is called the \emph{spherical diffraction measure} of $\nu$.
\end{definition}
In view of characterization (God2) of the Fourier transform of a measure, we have:
\begin{proposition} The spherical diffraction measure $\widehat{\eta}  \in M^+(\mathcal S^+(G, K))$ is uniquely determined by the fact that for all $f \in C_c(G, K)$ we have
\[\pushQED{\qed}\widehat{\eta}(|\widehat{f}|^2) = \widetilde{\eta}(f \ast f^*) = \|\mathcal P_\mu( {}_K f_K)\|^2_{L^2(\Omega_\mu^\times, \nu)} = \int_{\Omega_\mu^\times}\left|\int_{K\backslash G} {}_Kf(x) \, d\mu'(x) \right|^2 d\nu(\mu').\qedhere \popQED
\]
\end{proposition}

From the projection formula (Proposition \ref{ProjectionFormula}) we obtain immediately the following criterion for pure point spherical diffraction:
\begin{theorem}[Complete reducibility implies pure point spherical diffraction]\label{AbstractPurePoint} Assume that $(L^2(\Omega^\times_\mu, \nu), \pi_\nu)$ is spherically completely reducible in the sense that 
\[
L^2(\Omega^\times_\mu, \nu)^K = \bigoplus_{\omega \in \mathcal S^+(G, K)} L^2(\Omega^\times_\mu, \nu)^K_\omega.
\]
Then the spherical diffraction measure $\widehat{\eta}$ is given in terms of the diffraction coefficients $c_\nu(\omega)$ as
\[
\widehat{\eta} = \sum_{\omega \in \mathcal S^+(G, K)} c_\nu(\omega) \cdot \delta_{\check \omega},
\]
In particular, $\widehat{\eta}$ is a pure point measure.
\end{theorem}
\begin{proof} Let $f \in C_c(G, K)$ and recall that this implies that $\mathcal P_\mu({}_Kf_K) \in L^2(\Omega^\times_\mu, \nu)^K$. It thus follows from 
Proposition \ref{ProjectionFormula} that
\begin{eqnarray*}
\widehat{\eta}(|\widehat{f}|^2) &=&   \|\mathcal P_\mu({}_Kf_K)\|^2_{L^2(\Omega_\mu^\times, \nu)} \quad = \quad \sum_{\omega \in \mathcal S^+(G, K)}   \|{\rm proj}_\omega(\mathcal P_\mu({}_Kf_K))\|^2_{L^2(\Omega_\mu^\times, \nu)}\\
&=& \sum_{\omega \in \mathcal S^+(G, K)} c_\nu(\omega) \cdot |\widehat{f}(\check \omega)|^2.
\end{eqnarray*}
This shows that the measures $\widehat{\eta}$ and $\sum_{\omega \in \mathcal S^+(G, K)} c_\nu(\omega) \cdot \delta_{\check \omega}$ coincide on all functions of the form $|\widehat{f}|^2$ with $f \in C_c(G, K)$, and since these span a dense subspace of $C_c(\mathcal S^+(G, K))$, the theorem follows.
\end{proof}
In particular, the theorem applies if $L^2(\Omega^\times_\mu, \nu)$ is completely reducible as a unitary $G$-represen-tation. We will see in the next subsection that this is the case if $\mu$ is (the Dirac comb) of a weighted uniform regular model set and $\nu$ is the unique invariant measure on its hull. For non-uniform model sets, the representation $L^2(\Omega^\times_\mu, \nu)$ will not be completely reducible. In this case irreducible subrepresentations of $L^2(\Omega^\times_\mu, \nu)$ will provide some pure point spectrum, but there will also be continuous spectrum in the diffraction measure.
\subsection{Pure point spherical diffraction for weighted uniform regular model sets}
The goal of this subsection is to establish that weighted uniform model sets have pure point spherical diffraction. Thus let $\Lambda = \Lambda(G, H, \Gamma, W)$ be a uniform regular model set in $G$ and let $\pi_*\Lambda$ be the associated weighted model set in $K\backslash G$. Recall from \cite[Lemma~3.11]{BHP2} that $_K p$ induces a continuous $G$-factor map 
\[
\pi_*: \Omega_\Lambda \to \Omega_{\pi_*\Lambda},
\]
and that the unique $G$-invariant probability measure $\nu$ on $\Omega_{\pi_*\Lambda}$ is the push-forward under $\pi_*$ of the unique $G$-invariant probability measure $\widehat{\nu}$ on $ \Omega_\Lambda$. In particular, $\pi$ induces an embedding
\[
\pi^*: L^2(\Omega_{\pi_*\Lambda}, \nu) \hookrightarrow L^2(\Omega_{\Lambda}, \widehat{\nu}).
\]
In order to show that the spherical diffraction measure $\eta$ of $\nu$ is pure point, it suffices to show by Theorem \ref{AbstractPurePoint} that $L^2(\Omega_{\pi_*\Lambda}, \nu)$ is completely reducible. This is established in the following proposition.
\begin{proposition}[Complete reducibility for weighted uniform regular model sets]\label{ConcretePurePoint} The representation $L^2(\Omega_{\Lambda}, \widehat{\nu})$ is completely reducible with countable multiplicities, and hence the same holds for the subrepresentation $ L^2(\Omega_{\pi_*\Lambda}, \nu)$.
\end{proposition}
\begin{proof} We established in \cite{BHP1} that $L^2(\Omega_\Lambda, \widehat{\nu})$ is isomorphic to the space $L^2(Y, m_Y)$, where $Y := \Gamma \backslash (G \times H)$ and $m_Y$ denotes the unique $(G\times H)$-invariant probability measure on $Y$. Since $\Gamma$ is cocompact in $G\times H$, the $(G\times H)$-representation $L^2(Y)$ is completely reducible with finite multiplicities (see e.g. \cite[Thm. 7.2.5]{Wolf-07}). Since $(G,K)$ is a Gelfand pair, the group $G$ is of type $I$ (see e.g. \cite[Thm. 2.2]{Ciobotaru}). Consequently, every irreducible unitary representation $(G \times H)$-representation is of the form $V \boxtimes W$ where $V$ is an irreducible unitary $G$-representation, $W$ is an irreducible unitary $H$-representation and $V \boxtimes W$ is isomorphic to the completed tensor product of $V$ and $W$ with $(G\times H)$-action given by $(g,h).(v\otimes w) = gv\otimes hw$ (see e.g. \cite[Thm. 7.25]{Folland-95}). In this situation, if $(w_i)_{i \in I}$ is a Hilbert space basis of $W$ then, as $G$-representations,
\[
V \boxtimes W|_G \cong \widehat{\bigoplus_{i \in I}} V\otimes \mathbb C \cdot w_i \cong \widehat{\bigoplus_{i \in I}} V.
\]
Note that $I$ is countable, since $L^2(Y)$ and hence $W$ are separable. We deduce that, as $G$-representations, each $V \boxtimes W$ and thus also $L^2(Y)$ are completely reducible with countable multiplicities. 
\end{proof}
At this point we have established Theorem \ref{ppdiff}. The remainder of this article is devoted to a computation of the diffraction coefficients in various cases of interests.

\section{Diffraction coefficients of weighted uniform regular model sets} \label{sec:diffractionformula}

Throughout this section $\Lambda = \Lambda(G, H, \Gamma, W)$ denotes a uniform regular model set in $G$ (see  (see \cite[Def.\ 2.6]{BHP1})) constructed from a cut-and-project scheme $(G, H, \Gamma)$ (see \cite[Def.\ 2.3]{BHP1}) with window $W$. We denote by $\nu$ the unique $G$-invariant probability measure on $\Omega_{{}_Kp_*\delta_\Lambda}$ and by $\eta \in M^+(K\backslash G/K)$ its auto-correlation measure. We have seen in the previous section that the diffraction measure $\widehat{\eta} \in M^+(\mathcal S^+(G, K))$ is pure point. In this section we consider the problem of determining its coefficients in terms of the underlying lattice $\Gamma<G \times H$ and window $W\subset H$.

\subsection{The shadow transform}

We denote by $Y$ the homogeneous $(G \times H)$-space $Y := \Gamma\backslash (G\times H)$ and by $m_Y$ the unique $(G\times H)$-invariant probability measure on $Y$. We denote by ${}_\Gamma \pi_R$ the unitary $G$-representation
\[
{}_\Gamma \pi_R: G \to \mathcal U(L^2(Y, m_Y)), \quad ({}_\Gamma \pi_R(x)f)(\Gamma(g,h)) := f(\Gamma(gx, h))\]
as well as the corresponding $*$-representation ${}_\Gamma \pi_R: L^1(G) \to \mathcal B(L^2(Y, m_Y))$.  If $(V, \pi_V)$ is an irreducible spherical representation with spherical matrix coefficient $\omega = \omega_V \in \mathcal S^+(G, K)$, then we denote by \[
{\rm proj}_\omega:L^2(\Omega_\mu^\times, \nu) \to L^2(\Omega_\mu^\times, \nu)_\omega
\]
the projection onto the corresponding isotypical component and set
\[L^2(\Omega_\mu^\times, \nu)^K_\omega := L^2(\Omega_\mu^\times, \nu)^K \cap L^2(\Omega_\mu^\times, \nu)_\omega.\]
The countable subset
\[
{\rm spec}_{(G, K)}(\Gamma) := \{\omega \in \mathcal S^+(G, K) \mid L^2(\Omega_\mu^\times, \nu)_\omega \neq \{0\}\}
\]
of $\mathcal S^+(G, K)$ is called the \emph{$(G, K)$-spherical automorphic spectrum} of the lattice $\Gamma$.

\begin{remark}[Extending the periodization map to measurable functions]
For $M \in \{G, H, G \times H\}$ denote denote by $\mathcal L^\infty_c(M)$ the space of bounded measurable functions on $M$ which vanish outside a compact set. We then have a periodization map
\[
\mathcal P_\Gamma: \mathcal L^\infty_c(G \times H) \to L^2(Y, m_Y), \quad \mathcal P_\Gamma(F)(\Gamma(x,y)) = \sum_{\gamma \in \Gamma} F(\gamma(x,y))
\]
which extends the periodization map $\mathcal P_\Gamma: C_c(G \times H) \to C_0(Y)$ considered earlier. 
\end{remark}
We are going to show:
\begin{proposition}[Existence of the shadow transform]\label{ShadowExist}
For every $\omega \in {\rm spec}_{(G, K)}(\Gamma)$ and $r \in \mathcal L^\infty_c(H)$  there exists an element $\mathcal S_\Gamma(r)(\omega) \in L^2(Y, m_Y)_\omega^K$ such that for all $f \in C_c(G, K)$, 
\[
{\rm proj}_\omega(\mathcal P_\Gamma(f \otimes r)) = \widehat{f}(\check \omega) \cdot \mathcal S_\Gamma(r)(\omega).
\]
\end{proposition}
Collecting the constants $\mathcal S_\Gamma(r)(\omega)$ from Proposition \ref{ShadowExist} we can define a linear map
\[
\mathcal S_\Gamma: \mathcal L^\infty_c(H) \to \prod_{\omega \in {\rm spec}_{(G, K)}(\Gamma)}L^2(Y, m_Y)_\omega^K, \quad \mathcal S_\Gamma(r) := (\mathcal S_\Gamma(r)(\omega))_{\omega \in {\rm spec}_{(G, K)}(\Gamma)}.
\]
\begin{definition} The map $\mathcal S_\Gamma$ is called the \emph{shadow transform} of the lattice $\Gamma$.
\end{definition}
\begin{corollary}[$L^2$-norm of a periodization]\label{ShadowL2}  For every $f \in C_c(G, K)$ and $r \in \mathcal L^\infty_c(H)$ we have
\[\pushQED{\qed}
\|\mathcal P_\Gamma(f \otimes r)\|_2^2 = \sum_{\omega\in {\rm spec}_{(G, K)}(\Gamma)} \|{\rm proj}_\omega(\mathcal P_\Gamma(f \otimes r))\|_2^2 =  \sum_{\omega\in {\rm spec}_{(G, K)}(\Gamma)} |\widehat{f}(\check \omega)|^2 \|\mathcal S_\Gamma(r)(\omega)\|^2.\qedhere \popQED
\]
\end{corollary}
\begin{corollary}[Kernel of the shadow transform] A function $r \in \mathcal L^\infty_c(H)$ is contained in the kernel of the shadow transform if and only if $\mathcal P_\Gamma(f \otimes r) = 0$ almost everywhere for all $f \in C_c(G, K)$.\qed
\end{corollary}
The proof of Proposition \ref{ShadowExist} is in close analogy with the proof of the projection formula. Lemma \ref{ProjectionFormulaLemma1} and Corollary \ref{ProjectionFormulaCor1} translate into the current setting as follows:
\begin{lemma}\label{ShadowLemma1} For $\rho,f \in C_c(G, K)$ and $r \in \mathcal L^\infty_c(H)$ we have
\[
\mathcal P_\Gamma((\rho \ast f) \otimes r) = {}_\Gamma \pi_R(\check f) (\mathcal P_\Gamma(\rho \otimes r)).
\]
\end{lemma}
\begin{proof} By \eqref{Convolution} we have $\rho \ast f = \pi_R(\check f)(\rho)$, and hence for all $g \in G$ and $h \in H$ we have
\begin{eqnarray*}
\mathcal P_\Gamma((\rho \ast f) \otimes r)(\Gamma(g,h)) &=& \sum_{(\gamma_1, \gamma_2) \in \Gamma} \pi_R(\check f)(\rho)(\gamma_1g) r(\gamma_2 h)\\
&=&\int_G \check f(x) \sum_{(\gamma_1, \gamma_2) \in \Gamma} \rho(\gamma_1gx) r(\gamma_2h)\, dm_G(x)\\
&=& \int_G \check f(x) \mathcal P_\Gamma(\rho \otimes r)(\Gamma(gx, h))\, dm_G(x)\\
&=& {}_\Gamma \pi_R(\check f) (\mathcal P_\Gamma(\rho \otimes r))(\Gamma(g,h)),
\end{eqnarray*}
where we can exchange sum and integral since the sum is actually finite.
\end{proof}

\begin{corollary}\label{ShadowCor1} For $\rho, f \in C_c(G, K)$, $r \in \mathcal L^\infty_c(H)$ and $u \in L^2(\Omega_\mu^\times, \nu)^K_\omega$ we have
\[
\langle \mathcal P_\Gamma((\rho \ast f) \otimes r), u \rangle = \widehat{f}(\check \omega) \cdot  \langle \mathcal P_\Gamma(\rho \otimes r), u \rangle
\]
\end{corollary}
\begin{proof}Since $\pi_\nu$ is a $*$-representation, Lemma \ref{ShadowLemma1} yields
\[
\langle \mathcal P_\Gamma((\rho \ast f) \otimes r), u \rangle  = \langle{}_\Gamma \pi_R(\check f) (\mathcal P_\Gamma(\rho \otimes r)) ,u\rangle = \langle \mathcal P_\Gamma(\rho \otimes r) ,{}_\Gamma \pi_R(\check f)^*(u)\rangle
=\langle\mathcal P_\Gamma(\rho \otimes r) ,  {}_\Gamma \pi_R(\overline{f})(u)\rangle.
\]
By Lemma \ref{ActionByFourierCoefficient} and \eqref{FTfbar} we have
\[
{}_\Gamma \pi_R(\overline{f})u = \widehat{\overline{f}}(\omega)\cdot u = \overline{\widehat{f}(\check \omega)} \cdot u.
\]
The corollary follows.
\end{proof}
\begin{proof}[Proof of Proposition \ref{ShadowExist}] Let $(u_\alpha)_{\alpha \in I_\omega}$ be an orthonormal basis of $L^2(\Omega_\mu^\times, \nu)^K_\omega$. Since $\mathcal P_\Gamma(f \otimes r)$ is $K$-invariant, so is its projection onto $L^2(\Omega_\mu^\times, \nu)_\omega$, and thus 
\[
{\rm proj}_\omega(\mathcal P_\Gamma(f \otimes r)) =  \sum_{\alpha \in I_\omega} \langle \mathcal P_\Gamma (f \otimes r), u_\alpha\rangle u_\alpha.
\]
Now recall that our convenient approximate identify $(\rho_n)$ has been chosen so that $\rho_n \ast f \to f$ converges uniformly. This in turn implies that $\mathcal P_\Gamma((\rho_n \ast f)\otimes r) \to \mathcal P_\Gamma(f \otimes r)$ uniformly, and hence in $L^2$. We deduce with Corollary \ref{ShadowCor1} that
\begin{eqnarray*}
{\rm proj}_\omega(\mathcal P_\Gamma(f \otimes r)) &=& \sum_{\alpha \in I_\omega} \lim_{n \to \infty} \langle \mathcal P_\Gamma ((\rho_n \ast f) \otimes r), u_\alpha\rangle u_\alpha\\
&=&  \sum_{\alpha \in I_\omega} \lim_{n \to \infty} \langle \widehat{f}(\check \omega) \cdot  \langle \mathcal P_\Gamma(\rho_n \otimes r), u_{\alpha} \rangle u_\alpha\\
&=& \widehat{f}(\check \omega) \cdot \sum_{\alpha \in I_\omega} \lim_{n \to \infty} \langle \langle \mathcal P_\Gamma(\rho_n \otimes r), u_{\alpha} \rangle u_\alpha.
\end{eqnarray*}
Since the second factor is independent of $f$, the proposition follows.
\end{proof}
\subsection{The diffraction formula}
Recall from Theorem \ref{AbstractPurePoint} that the diffraction measure $\widehat{\eta}$ of the unique $G$-invariant measure $\nu$ on $\Omega_{\pi_*\Lambda}$ is of the form
\[
\widehat{\eta} = \sum_{\omega \in \mathcal S^+(G, K)} c(\omega)\cdot \delta_{\check \omega}.
\]
We can now express the diffraction coefficients $c(\omega)$ in terms of the Shadow transform of the underlying lattice $\Gamma$ and the characteristic function ${\bf 1}_W$ of the underlying window:
\begin{theorem}[Diffraction formula for weighted model sets]\label{DiffAbstract} The diffaction measure is given by the formula
\[
\widehat{\eta} = \sum_{\omega \in {\rm spec}_{(G, K)}(\Gamma)} \|\mathcal S_\Gamma({\bf 1}_W)(\omega)\|^2\cdot \delta_{\check \omega}.
\]
Thus $c(\omega) = \|\mathcal S_\Gamma({\bf 1}_W)(\omega)\|^2$ if $\omega \in  {\rm spec}_{(G, K)}(\Gamma)$ and $c(\omega) = 0$ otherwise.
\end{theorem}
\begin{proof} By \cite[Corollary~4.18]{BHP2} we have for every $f \in C_c(G,K)$
\[
\widehat{\eta}(\widehat{|f|^2}) = {\eta}({}_Kf_K\ast ({}_Kf_K)^*) = \|\mathcal P_\Gamma (f \otimes {\bf 1}_W)\|^2_{L^2(\Gamma\backslash(G \times H))}.
\]
Combining this with Corollary \ref{ShadowL2} we obtain
\[
\widehat{\eta}(|\widehat{f}|^2) =  \sum_{\omega\in {\rm spec}_{(G, K)}(\Gamma)} |\widehat{f}(\check \omega)|^2 \|\mathcal S_\Gamma({\bf 1}_W)(\omega)\|^2 =  \left(\sum_{\omega\in {\rm spec}_{(G, K)}(\Gamma)} \|\mathcal S_\Gamma({\bf 1}_W)(\omega)\|^2 \cdot \delta_{\check \omega}\right)(|\widehat{f}|^2).
\]
The theorem follows.
\end{proof}
In praxis, the the diffraction coefficients $c(\omega) = \|\mathcal S_\Gamma({\bf 1}_W)(\omega)\|^2$ are usually much easier to determine than the shadow transform itself. For example they admit the following characterization:
\begin{proposition}\label{DiffCoeffConvenient} The diffraction coefficients $c(\omega)$ are uniquely determined by the fact that for all $f \in C_c(G, K)$,
\[
 \sum_{(\gamma_1,\gamma_2) \in \Gamma} (f \ast f^*)(\gamma_1) ({\bf 1}_W \ast {\bf 1}_{W^{-1}})(\gamma_2) = \sum_{\omega \in \mathcal S^+(G, K)} c(\omega) |\widehat f(\check \omega)|^2.
\]
\end{proposition}
\begin{proof}
By Corollary~4.18 and Proposition~4.19 in \cite{BHP2} we have for all $f \in C_c(G, K)$, 
\[
\widehat{\eta}(|\widehat{f}|^2) = \eta(f \ast f^*) = \| \mathcal P_\Gamma(f \otimes {\bf 1}_W))\|^2 = \sum_{(\gamma_1,\gamma_2) \in \Gamma}
(f^* * f)(\gamma_1) ({\bf 1}_W \ast {\bf 1}_{W^{-1}})(\gamma_2).
\]
Since $\{f \ast f^* \mid f \in C_c(G, K)\}$ spans a dense subspace in $C_c(G, K)$, this determines $\widehat{\eta}$.

\end{proof}

\subsection{The shadow transform as a generalized Hecke correspondence}\label{SecHecke} 
In addition to the standing assumptions of this section assume now that $G$ is a totally disconnected lcsc group and $K<G$ is a compact open subgroup. In this case the shadow transform is closely related to a more classical transform in harmonic analysis, the so-called \emph{Hecke correspondence}, which we recall briefly.

Denote by $p_G: G \times H \to G$ and $p_H: G \times H \to H$ the canonical projections. We consider a lattice $\Gamma < G\times H$ such that
\begin{itemize}
\item the projections $\Gamma_G := \pi_G(\Gamma)$ and $\Gamma_H := \pi_H(\Gamma)$ are dense in $G$ and $H$ respectively;
\item $p_G|_\Gamma: \Gamma \to \Gamma_G$ is bijective.
\end{itemize}
We then denote by $\tau: \Gamma_G \to H$ the homomorphism $\tau(g)= p_H((p_G|_\Gamma)^{-1}(g))$. For the moment we do not assume that $\Gamma$ is uniform. Since $\Gamma_G$ is dense in $G$ and $K$ is open, the multiplication map $\Gamma_G\times K \to G$ is onto. We denote by $g \mapsto (\gamma_g, k_g)$ a fixed Borel section of this map. For simplicity let us normalize the Haar measure on $G$ such that $m_G(K) = 1$. 
\begin{proposition}[Hecke correspondence] Let $\Gamma_K := \Gamma_G \cap K$ and $\Gamma_0 := \tau(\Gamma_K) < H$.
\begin{enumerate}[(i)]
\item $\Gamma_0 < H$ is a lattice, which is uniform if and only if $\Gamma$ is uniform.
\item The map $j:\Gamma \backslash (G\times H)/K \to  \Gamma_0\backslash H$ given by $j(\Gamma(g,h)K):= \Gamma_0 \tau(\gamma_g^{-1})h$ is a homeomorphism with inverse given by
$i:  \Gamma_0\backslash H \to \Gamma \backslash (G\times H)/K$, $\Gamma_0 h \mapsto \Gamma(e, h)K$. 
\item $i$ and $j$ induce mutually inverse isomorphisms of $H$-representations
\[
 i^*:L^2(\Gamma \backslash (G \times H))^K \to L^2(\Gamma_0\backslash H) \quad \text{and} \quad j^*:  L^2(\Gamma_0\backslash H) \to L^2(\Gamma \backslash (G \times H))^K.
\]
\item The Hecke algebra $C_c(G, K)$ acts on $L^2(\Gamma_0\backslash H)$ via
\[
T(\rho)(f)(\Gamma_0h) = \int_G \rho(g) f(\Gamma_0 \tau(\gamma_g)^{-1}h) \, dm_G(g) \quad(\rho \in C_c(G, K), \; f \in  L^2(H/\Gamma_0)).
\]
\end{enumerate}
\end{proposition}
\begin{proof} We first prove (ii). Observe first that for all $(g,h) \in G \times H$,
\[
\Gamma(g,h)K = \Gamma(\gamma_gk_g, h) K = \Gamma(\gamma_g, h)K = \Gamma (e, \tau(\gamma_g)^{-1}h)K.
\]
This shows that the map $q: H \to \Gamma \backslash (G\times H)/K$ given by $q(h) := \Gamma(e,h)K$ is onto. Now assume that $q(h_1) = q(h_2)$. Then
\begin{equation}\label{CompactOpenMove}
\exists k \in K, \gamma \in \Gamma_G: \;(e,h_1) = (\gamma, \tau(\gamma))(e,h_2)k = (\gamma k, \tau(\gamma)h_2)
\end{equation}
This implies that $k\gamma = e$, hence $k = \gamma^{-1} \in \Gamma_K$, and thus $h_1 = \tau(\gamma)h_2 \in \Gamma_0 h_2$. Conversely, if $h_1 \in \Gamma_0h_2$, then $q(h_1) = q(h_2)$. Thus $q$ factors through a continuous bijection $i$ as in the proposition with inverse $j$. Now note that $H$ acts on $\Gamma \backslash (G\times H)/K$ from the right, since it commutes with $K$, and that $i$ is $H$-equivariant. It follows that $i$ is open, whence $i$ and $j$ are mutually inverse homeomorphisms. This proves (ii) and shows in particular that $\Gamma_0$ is of finite covolume, respectively cocompact in $H$ if and only if $\Gamma< G\times H$ has the corresponding property. To show (i) it thus remains to show only that $\Gamma_0$ is discrete. However, for every compact subset $W \subset H$ we have
\[
\Gamma_0 \cap W = \tau(\Gamma_K) \cap W = \tau(p_G((K \times W)\cap \Gamma)),
\]
which is finite by discreteness of $\Gamma$. This finishes the proof of (i) and provides us with a unique $H$-invariant probability measure $m_{\Gamma_0\backslash H}$ on $\Gamma_0\backslash H$. Now (ii) yields an $H$-equivariant isomorphism $i^*:  C_c(\Gamma \backslash (G \times H))^K  \to C_c(\Gamma_0\backslash H)$, and under this identification the unique $H$-invariant measures on $\Gamma\backslash (G \times H)$ and $\Gamma_0\backslash H$ must correspond, hence (iii) holds. In particular, if $\pi$ denote the unitary representation $\pi: G \to \mathcal U(L^2(\Gamma\backslash (G \times H)))$, then $C_c(G, K)$ acts on $L^2(\Gamma_0\backslash H)$ via 
\[
T(\rho)(f) :=  i^*(\pi(\rho).(j^*f)) \quad (\rho \in C_c(G, K), f \in L^2(\Gamma_0\backslash H)).
\]
Writing out the definitions of $i^*, j^*$ and $\pi$ explicitly we end up with (iv).
\end{proof}
If we assume now that $\Gamma$ is cocompact, then $L^2(\Gamma_0 \backslash H)$ decomposes under the action of the Hecke algebra as
\[
L^2(\Gamma_0 \backslash H) = \bigoplus_{\omega \in {\rm spec}_{(G,K)}(\Gamma)}L^2(\Gamma_0 \backslash H)_\omega,
\]
and we note by ${\rm proj}_\omega: L^2(\Gamma_0 \backslash H) \to L^2(\Gamma_0 \backslash H)_\omega$ the canonical projection.
\begin{proposition}[Shadow transform vs. Hecke correspondence]\label{ShadowVSHecke} Let $r \in \mathcal L^\infty_c(\Gamma_0\backslash H)$ and denote by $j^*:  L^2(\Gamma_0\backslash H) \to L^2(\Gamma\backslash(G\times H))^K$ the Hecke correspondence. Then for all $\omega \in {\rm spec}_{(G,K)}(\Gamma)$,
\[
\mathcal S_\Gamma r(\omega) = j^*{\rm proj}_\omega(\mathcal P_{\Gamma_0} r) \qand \|\mathcal S_\Gamma r(\omega)\|^2 = \|{\rm proj}_\omega(\mathcal P_{\Gamma_0} r) \|^2.
\]
\end{proposition}
\begin{proof} The key observation is that since $K$ is compact-open, ${\bf 1}_K$ is a two-sided identity in $C_c(G, K)$. Consequently, for all $h \in C_c(G, K)$ and $\omega \in \mathcal S^+(G, K)$,
\[
\widehat{{\bf 1}}_K(\omega) \cdot \widehat{h}(\omega) = \widehat{{\bf 1}_K \ast h}(\omega) = \widehat{h}(\omega),
\]
and hence $\widehat{{\bf 1}}_K(\omega) = 1$. By definition of the shadow transform we have
 \[
\widehat{f}(\check \omega) \cdot \mathcal S_\Gamma(r)(\omega) = {\rm proj}_\omega(\mathcal P_\Gamma(f \otimes r)) 
\]
for every $f \in C_c(G, K)$. If we choose $f := {\bf 1}_K$, then we obtain
 \[
 \mathcal S_\Gamma(r)(\omega) =  \widehat{{\bf 1}}_K(\check \omega) \cdot \mathcal S_\Gamma(r)(\omega)  = {\rm proj}_\omega(\mathcal P_\Gamma({\bf 1}_K \otimes r)).
\]
Now for $(g,h) \in G \times H$ we have
\begin{eqnarray*}
\mathcal P_\Gamma({\bf 1}_K \otimes r)(\Gamma(g,h)) &=& \sum_{\gamma \in \Gamma_G} {\bf 1}_K(\gamma \gamma_g k_g)r(\tau(\gamma)h) \quad = \quad \sum_{\gamma' = \gamma\gamma_g \in \Gamma_G} {\bf 1}_K(\gamma') r(\tau(\gamma')\tau(\gamma_g^{-1})h)\\
&=& \sum_{\gamma_0 \in \Gamma_0} r(\gamma_0 \tau(\gamma_g^{-1})h) \quad = \quad \mathcal P_{\Gamma_0} r(\Gamma_0\tau(\gamma_g^{-1})h)\\
&=&  j^*(\mathcal P_{\Gamma_0}r)(\Gamma(g,h)),
\end{eqnarray*}
hence $ \mathcal S_\Gamma(r)(\omega) =  {\rm proj}_\omega( j^*(\mathcal P_{\Gamma_0}r))$, and since $j^*$ is equivariant under the action of the Hecke algebra, the proposition follows.
\end{proof}
If we denote by $\widehat{\eta}$ the diffraction measure of the unique $G$-invariant measure on $\Omega_{\pi_*\Gamma}$, then we obtain:
\begin{corollary}[Spherical diffraction formula for a compact-open $K$] The diffraction measure $\widehat{\eta}$ is given by the formula\[
\pushQED{\qed}\widehat{\eta} = \sum_{\omega \in {\rm spec}_{(G, K)}(\Gamma)}\|{\rm proj}_\omega(\mathcal P_{\Gamma_0} {\bf 1}_W) \|^2\cdot \delta_{\check \omega}.\qedhere\popQED
\]
\end{corollary}

\subsection{Classical examples} 
In the case where $G$ and $H$ are abelian and $K = \{e\}$, the diffraction formula in Theorem \ref{DiffAbstract} reduces to \cite[Thm. 9.4]{BaakeG-13}, which in its essence goes back to the pioneering work of Meyer \cite{Meyer-70, Meyer-95}. Let us briefly explain this reduction:

Let $\Lambda = \Lambda(G, H, \Gamma, W)$ be a uniform regular model set and assume that $G$ and $H$ are abelian.
Denote by $\widehat{G}$ and $\widehat{H}$ the dual groups of $G$ and $H$ respectively, and identify the dual group of $G \times H$ with $\widehat{G} \times \widehat{H}$. We define the \emph{dual lattice} of $\Gamma$ by
\[
\Gamma^\perp := \{(\xi_1, \xi_2) \in \widehat{G} \times \widehat{H} \mid \forall (\gamma_1, \gamma_2) \in \Gamma: \; \xi_1(\gamma_1) \xi_2(\gamma_2) = 1\} < \widehat{G} \times \widehat{H}.
\]
By assumption $\Gamma$ projects injectively to $G$ and densely to $H$. As in \cite[p.19]{MoodySurvey} one deduces:
\begin{lemma}\label{CPPropertyAbelian} The dual lattice $\Gamma^\perp$ projects injectively to $\widehat{G}$ and densely to $\widehat{H}$.
\end{lemma}
\begin{proof} If $\chi \in \Gamma^\perp$ is contained in the kernel of the projection to $\widehat{G}$, i.e. $\chi = \chi_1 \otimes \chi_2$, then $\chi_2$ is trivial on the projection of $\Gamma$ to $H$, hence on all of $H$ by continuity, and thus $\chi=1$. Moreover, since $\Gamma$ projects injectively to $G$, the map $H \to (G\times H)/\Gamma$ is injective, and hence the dual map $\Gamma^\perp \to H^\perp$ has dense image.
\end{proof}
We denote by $\Gamma^\perp_{\widehat{G}}$ and $\Gamma^\perp_{\widehat{H}}$ the images of $\Gamma^\perp$ under the canonical projections $p_{\widehat{G}}: \widehat{G} \times \widehat{H} \to \widehat{G}$ and  $p_{\widehat{H}}: \widehat{G} \times \widehat{H} \to \widehat{H}$ respectively. Using the lemma we may define
\[
\zeta:= p_{\widehat{H}} \circ (p_{\widehat{G}}|_{\Gamma^\perp})^{-1}: \Gamma^\perp_{\widehat{G}} \to \widehat{H}, 
\]
so that $(\xi, \zeta(\xi)) \in \Gamma^\perp$ for all $\xi \in \Gamma^\perp_{\widehat{G}}$.

Given $f \in L^1(G)\cap L^2(G)$ we denote by
\[
\widehat{f}(\chi_1) := \int_G f(g)\chi_1(g^{-1})dm_G(g)
\]
the Fourier transform of $f$ with respect to $m_G$. We normalize the Haar measure $m_H$ on $H$ so that $\Gamma$ has covolume $1$ in $G \times H$ and use the same symbol to denote the Fourier transform with respect to $m_H$. We denote by $m_Y$ the corresponding Haar probability measure on $Y := \Gamma \backslash (G \times H)$.

Every $\xi \in \Gamma^\perp$ defines a $\Gamma$-invariant function on $G \times H$, hence descends to a function ${}_\Gamma \xi$ on $Y$, and the functions $\{{}_\Gamma \xi \mid \xi \in \Gamma^\perp\}$ form an orthonormal basis of $L^2(Y, m_Y)$. We deduce that for $f \in C_c(G)$,
\[
\|\mathcal P_\Gamma(f \otimes {\bf 1}_W)\| = \sum_{\xi \in \Gamma^\perp} |\langle \mathcal P_\Gamma(f \otimes {\bf 1}_W), {}_\Gamma \xi\rangle_{L^2(Y, m_Y)}|^2.
\]
Now if $\mathcal F$ denotes a fundamental domain for $\Gamma$ in $G \times H$, then for $\xi = (\xi_1, \xi_2) \in \Gamma^\perp$ we have
\begin{eqnarray*}
 \langle \mathcal P_\Gamma(f \otimes {\bf 1}_W), {}_\Gamma \xi\rangle_{L^2(Y, m_Y)} &=& \int_{\mathcal F} \sum_{(\gamma_1, \gamma_2 \in \Gamma)} f(\gamma_1 g) {\bf 1}_W(\gamma_2 h) \overline{\xi(g,h)}\, dm_G(g) dm_H(h)\\ &=& \sum_{\gamma \in \Gamma)} \int_{\gamma\mathcal F} f(g){\bf 1}_W(h)\overline{\xi_1(g)}\overline{\xi_2(h)}\, dm_G(g) dm_H(h)\\
 &=& \widehat{f}(\xi_1) \widehat{{\bf 1}}_W(\xi_2).
\end{eqnarray*}
We thus obtain
\[
\sum_{(\gamma_1,\gamma_2) \in \Gamma}
(f * f^*)(\gamma_1) ({\bf 1}_W * {\bf 1}_{W^{-1}})(\gamma_2) = \|\mathcal P_\Gamma(f \otimes {\bf 1}_W)\|  = \sum_{(\xi_1,\xi_2) \in \Gamma^\perp}|\widehat{f}(\xi_1)|^2|\widehat{{\bf 1}}_W(\xi_2)|^2.
\]
Note that this is not just a formal consequence of the Poisson summation formula, since ${\bf 1}_W$ is not smooth, but it is similar in spirit. In any case, we obtain from Proposition \ref{DiffCoeffConvenient} the formula
\[
\widehat{\eta} = \sum_{\xi \in \Gamma^\perp_{\widehat{G}}} |\widehat{{\bf 1}}_W(\zeta(\xi))|^2 \cdot \delta_\xi,
\]
which is Meyer's formula. 

We now discuss an extension of this formula which appears (with different notation and under the name of ``powder diffraction'') in \cite{BFG}. For this let $N = \R^d$, $H = \R^m$ and $K = O(d)$, so that $G = K \ltimes N$ is the isometry group of the Euclidean plane, and $(G, K)$ is a Gelfand pair with $K\backslash G = \R^n$. Given a character $\xi \in \widehat{N}$ we denote by $q_\xi$ the Bessel function
\[
q_\xi: N \to \C, \quad q_\xi(n) := \int_K \xi(k.n)dm_K(k),
\]
and extend it to a positive-definite spherical function $\omega_\xi: G \to \C$ on $G$ by $\omega_\xi(k_o, n) :=q_\xi(n)$. We then obtain an identification
\[
K\backslash \widehat{N} \to \mathcal S^+(G, K), \quad K\xi \to \omega_\xi. 
\]

The spherical Fourier transform of $(G, K)$ relates to the usual Fourier transform of $N$ as follows. We have an isomorphism
\[
\iota:  C_c(N)^K \to C_c(G, K), \quad \iota(f)((k, n)) = f(n),
\]
and if $f \in C_c(N)^K$ with Fourier transform $\widehat{f} \in L^2(\widehat{N})$, then
\begin{eqnarray*}
\widehat{f}(\xi) &=& \int_N f(n)\overline{\xi(n)} dm_N(n) \quad = \quad  \int_N \int_K f(k^{-1}.n) dm_K(k) \overline{\xi}(n) dm_N(n)\\
 &=&  \int_N f(n)  \overline{\xi(k.n)} dm_K(k) dm_N(n) \quad = \quad \langle f, q_\xi \rangle \quad = \widehat{\iota(f)}(\omega_\xi).
\end{eqnarray*}
Now let $\Gamma_o < N \times H$ be a lattice which projects injectively to $N$ and densely to $H$; then the image 
\[
\Gamma = \big\{ ((e, \gamma_1),\gamma_2) \in (K \ltimes N) \times H \, : \, (\gamma_1,\gamma_2) \in \Gamma_o \big\},
\]
of $\Gamma_o$ in $G\times H$ is a lattice, and $(G, H, \Gamma)$ is a cut-and-project scheme. We now pick a regular window $W$ in $H$ so that $\widetilde{\Lambda} = \widetilde{\Lambda}(G, H, \Gamma, W)$ is a regular model set and $\Lambda = p_{K\backslash G}(\widetilde{\Lambda})$ is a regular model set in $K\backslash G = \R^d$. 

Denote by $\Gamma_o^\perp \subset \widehat{N} \times \widehat{H}$ the dual lattice of $\Gamma_o$ and by $(\Gamma_o^\perp)_{\widehat{N}}$ ist projection to $\widehat{N}$. From Lemma \ref{CPPropertyAbelian} we obtain a map $\zeta:(\Gamma_o^\perp)_{\widehat{N}} \to \widehat{H}$ such that
\[
(\xi, \zeta(\xi)) \in \Gamma_o^\perp \quad \text{for all } \xi \in (\Gamma_o^\perp)_{\widehat{N}}.
\]
For $f \in C_c(G, K)$ and $f_o = \iota^{-1}(f) \in C_c(N)^K$ the computation in the abelian case yields
\begin{eqnarray*}
\sum_{((e, \gamma_1), \gamma_2) \in \Gamma} (f \ast f^*)(e,\gamma_1) ( {\bf 1}_W\ast {\bf 1}_{W^{-1}})(\gamma_2) &=&
 \sum_{(\gamma_1,\gamma_2) \in \Gamma_o}
(f_o * f^*_o)(\gamma_1)  ( {\bf 1}_W\ast {\bf 1}_{W^{-1}})(\gamma_2)\\
&=& \sum_{(\xi_1,\xi_2) \in \Gamma_o^{\perp}}
|\widehat{f}_o(\xi_1)|^2 |\widehat{{\bf 1}}_W(\xi_2)|^2
\end{eqnarray*}
and since $\widehat{f_o}(\xi_1) = \widehat{f}(\omega_{\xi_1})$ we deduce from Proposition \ref{DiffCoeffConvenient} that
\[
\widehat{\eta} =  \sum_{\xi \in  (\Gamma_o^\perp)_{\widehat{N}}} \left( \sum_{\chi \in K.\xi \cap  (\Gamma_o^\perp)_{\widehat{N}}} |\widehat{\bf 1}_W(\zeta(\chi))|^2\right)\delta_\xi,
\]
which can be seen as a version of the ``spherical Poisson summation formula'' \cite{BFG} for regular model sets.

\subsection{Non-classical examples}
Since the formulas in the previous subsection were already known, the question arises to which other classes of examples our general diffraction formula can be applied to. In order to run our machinery we need:
\begin{enumerate}
\item a Gelfand pair $(G, K)$ with $K$ compact;
\item another lcsc group $H$ such that $G \times H$ admits a lattice $\Gamma$ which projects injectively to $G$ and densely to $H$.
\end{enumerate}
For simplicity let us also assume that
\begin{enumerate}
\setcounter{enumi}{2}
\item $G$ and $H$ are connected Lie groups and the manifold $G/K$ is simply-connected.
\end{enumerate}
There are two main sources of examples for such quadruples $(G, K, H, \Gamma)$. Let us first consider the case where $G$ is amenable. In this case we have the following result:
\begin{proposition}[Cut-and-project schemes of amenable Lie groups]\label{AmenableLie} Assume that $(G, K, H, \Gamma)$ satisfy (1)-(3) above and that $G$ is amenable. If $G$ acts effectively on $G/K$, then the following hold:
\begin{enumerate}[(i)]
\item $G = N \rtimes L$, where $L$ is compact and contains $K$, and $N$ is either abelian or $2$-step nilpotent.
\item If $\pi_G(\Gamma)$ is contained in $N$, then $H$ is either abelian or $2$-step nilpotent and $\Gamma$ is uniform.
\end{enumerate}
\end{proposition}
\begin{proof} (i) is an immediate consequnce of Vinberg's decomposition theorem \cite[Thm. 13.3.20]{Wolf-07}. If $\pi_G(\Gamma)<N$, then $\pi_G(\Gamma)$ and thus $\Gamma$ are $2$-step nilpotent. But then also $H$ is $2$-step nilpotent since it contains the dense $2$-step nilpotent subgroup $\pi_H(\Gamma)$, and thus $\Gamma$ is cocompact by \cite[Thm. 2.1]{Raghunathan}.
\end{proof}
The assumption that $G$ acts effectively is not essential and can always be arranged by passing to a quotient. The crucial point is that assuming amenablity of $G$, the possible pairs $(G, L)$ appearing in (i) can actually be classified; these are called  \emph{nilmanifold pairs} and all arise essentially from the 23 families of ``maximal irreducible'' nilmanifold pairs listed in \cite[Sec. 13.4]{Wolf-07}. 

If we add the additional assumption that
\begin{enumerate}
\setcounter{enumi}{3}
\item $\pi_G(\Gamma)<N$,
\end{enumerate}
then the quadruples $(G, K, H, \Gamma)$ satisfying (1)-(4) with $G$ amenable can actually be completely classified. Namely, $H$ has to be a $2$-step nilpotent Lie group, and these are well-known. Then $\Gamma$ has to be a lattice in the nilpotent Lie group $N \times H$, and hence arises from a rational basis of the Lie algebra of $N \times H$ by the construction described in \cite[Remark after Thm. 2.12]{Raghunathan}. In each of these cases, one can try to compute explicitly the diffraction formula in terms of a given regular window $W$. We will carry this out for Heisenberg motion groups in Section \ref{SecHeisenberg}. While many of the ideas work rather generally for quadruples satisfying (1)-(4) above, some of our estimates are specific to the Heisenberg group. (For example, we use square-integrability of irreducible spherical representations in an essential way.)

If we drop the condition that $\pi_G(\Gamma)<N$ then we can no longer classify the corresponding cut-and-project sets. Note that if we drop the assumption that $G$ and $H$ are connected, then we can no longer even guarantee that $\Gamma$ is uniform. In fact, there exists a cut-and-project scheme $(G, H, \Gamma)$ and a compact subgroup $K<G$ such that $G$ and $H$ are compact-by-abelian (in particular amenable), $(G, K)$ is a Gelfand pair and $\Gamma$ is non-uniform, see \cite[p. 8]{BHP1} which is based on \cite[Example 3.5]{BenoistQuint} due to Bader, Caprace, Gelander and Monod. This shows that we are very far from classifying the possible cut-and-project sets in a commutative space $G/K$ with $G$ a general lcsc amenable group.

Among the examples of quadruples satisfying (1)-(3) with non-amenable $G$, a central role is played by semisimple (or more generally, reductive) Gelfand pairs. Note that if $G$ is any semisimple Lie group with finite center and $K$ a maximal compact subgroup, then $(G, K)$ is a Gelfand pair, and $K\backslash G$ is a Riemannian symmetric space. In this case, there always exist both uniform and non-uniform cut-and-project schemes of the form $(G, H, \Gamma)$ satisfying (1)-(3). For example, one can always take $H = G$ or $H = G_\C$, the complexification of $G$. The assumption that $\Gamma$ projects densely onto $H$ implies by Margulis' arithmeticity theorem that $\Gamma$ is an arithmetic group. This implies that all weighted model sets in Riemannian symmetric spaces are of arithmetic origin. The simplest examples arise for $G = {\rm SL}_2(\R)$, and we will discuss this case in Section \ref{SecSemisimple}.

\section{Virtually nilpotent examples: Heisenberg motion groups}\label{SecHeisenberg}

\subsection{Spherical functions for Heisenberg motion groups}
Given $d \geq 1$ we abbreviate $V_d := \C^d$ and denote by $\langle \cdot, \cdot \rangle$ and by $m_{V_d}$ the Lebesgue measure on $V_d$. The \emph{standard symplectic form} $\beta_d$ on $V_d$ is given in terms of the standard Hermitian inner product  $\langle \cdot, \cdot \rangle$ by the formula 
\[
\beta_d(u,v) = -\frac{1}{2} \im \langle u,v \rangle, \quad \textrm{for $u,v \in V_d$}.
\]
\begin{definition}
The $(2d+1)$-dimensional \emph{Heisenberg group} $N_d = \bR \oplus_{\beta_d} V_d$ is the group with underlying set $\bR \times V_d$ and multiplication is given by the formula \[
(s,u)(t,v) = (s+t+\beta_d(u,v),u+v), \quad \textrm{for $(s,u), (t,v) \in N_d$}.
\]
\end{definition}
Since $\beta(v,v) = 0$ for all $v \in V$, we have $(t,v)^{-1} = (-t,-v)$ for all $(t,v) \in N$. The Haar measure on $N_d$ is given by $m_{N_d} = m_{\bR} \otimes m_{V_d}$, and is clearly both left- and right-invariant. 
\begin{remark}[Heisenberg motion group]
The group $K_d^{\rm max} := U(d)$ acts on $N_d$ by automorphisms via $k.(t,u) = (t, ku)$. If $K$ is any closed subgroup of $K_d^{\rm max}$ containing the diagonal subgroup $K_d:= \bT^d$, then $(K \ltimes N_d, K)$ is  a Gelfand pair \cite[Corollary~13.2.3]{Wolf-07}, and $K\ltimes N_d$ is called a \emph{Heisenberg motion group}. We will focus on the \emph{minimal Heisenberg motion group} $G_d := K_d \ltimes N_d$. Bi-invariant functions for this Gelfand pair correspond to \emph{polyradial} functions on the Heisenberg group, whereas bi-invariant functions for $K_d^{\rm max}$ are given by the much smaller space of \emph{radial} functions. Correspondingly, our polyradial diffraction formula is stronger than the corresponding radial diffraction formula. To obtain the latter from the former one basically has to expand the spherical functions of the larger Gelfand pair into those of the smaller Gelfand pair using the formulas from \cite{Th}. We omit the details.
\end{remark}
\begin{remark}[Notation concerning function spaces] For every $d \in \mathbb N$ we have an isomorphism of $*$-algebras $\iota_d: C_c(N_d)^{K_d} \to C_c(G_d, K_d)$ which is given by $\iota_d(f)(k,(t, v)) = f(t,v)$ and $\iota_d^{-1}f(t,v) = f(e,(t,v))$. It extends continuously to all $L^p$-spaces for $1\leq p < \infty$ and preserves smooth functions.
\end{remark}
We now recall the classification of spherical functions for the Gelfand pair $(G_d, K_d) = (\bT^d \ltimes (\R \oplus_{\beta_d} V_d), \bT^d)$ associated with the minimal Heisenberg motion group for a fixed $d \in \mathbb N$. We need the following notion:
\begin{definition} Let $\varphi, \psi \in L^1(V_d)$ and $\tau \in \bR$. The function 
\begin{equation}
\label{def_twisted}
(\varphi *_\tau \psi)(v) = \int_{V_d} \varphi(u) \psi(v-u) e^{-i \tau \beta(u,v)} \, dm_{V_d}(u).
\end{equation}
is called the \emph{$\tau$-twisted convolution} of $\varphi$ and $\psi$.
\end{definition}

It is not hard to see that $\varphi *_\tau \psi \in L^1(V)$ and it follows from $K_d$-invariance of $m_{V_d}$ that
\[
L^1(V_d)^{K_d} *_\tau L^1(V_d)^{K_d} \subset L^1(V_d)^{K_d}  \quad \textrm{for every $\tau \in \R$.}
\]
The following result is significantly harder (see \cite[Proposition 1.3.4]{Th}).
\begin{proposition}
For every $\tau \neq 0$ we have $L^2(V) *_\tau L^2(V) \subset L^2(V)$ and thus 
\[
\pushQED{\qed} 
L^2(V_d)^{K_d} *_\tau L^2(V_d)^{K_d} \subset L^2(V_d)^{K_d} \quad \textrm{for every $\tau \in \R \setminus\{0\}$.}
\qedhere
\popQED
\]
\end{proposition}
Note that for $\tau = 0$ the convolution $\ast_0$ is just the usual convolution, and hence the proposition does not extend to the case $\tau = 0$. The relation between twisted convolution and the Heisenberg group is as follows. If $f \in L^1(N_d)$, then we define the \emph{$\tau$-central Fourier transform} $f_\tau \in L^1(V_d)$ as
\[
f_\tau(v) = \int_{\bR} f(t,v) e^{-i\tau t} \, dm_{\bR}(t).
\]
In other words, if $f_v(t) := f(t, v)$, then $f_\tau(v) = \widehat{f_v}(\tau)$ is the Fourier transform in the central variable evaluated at $\tau$. We observe:
\begin{lemma}
\label{lemma_cft}
For every $\tau \in \R$ and all $f, g \in L^1(N_d)$ we have $(f * g)_\tau = f_\tau *_\tau g_\tau$.
\end{lemma}

\begin{proof}
First note that if $f,g \in L^1(N_d)$, then
\begin{eqnarray*}
(f * g)(t,v) 
&=& 
\int_{N_d} f(s,u) g((s,u)^{-1}(t,v)) \, dm_{N_d}(s,u) \\
&=&
\int_{{N_d}} f(s,u) g(t-s-\beta(u,v),v-u) \, dm_{N_d}(s,u).
\end{eqnarray*}
Hence, for every $\tau$,
\begin{eqnarray*}
(f * g)_\tau(v) 
&=& 
\int_{\bR} \Big( \int_{{N_d}} f(s,u) g(t-s-\beta(u,v),v-u) \, dm_{N_d}(s,u) \Big) e^{-i\tau t} \, dm_{\bR}(t) \\
&=&
\int_{{V_d}} \Big( \int_{\bR} \int_{\bR} f(s,u) e^{-i\tau s} \, g(t,v-u) e^{-i\tau t} \, dm_\bR(s) \, dm_{\bR}(t) \Big)\, e^{-i \tau \beta(u,v)} \, dm_{{V_d}}(u) \\
&=&
\int_{V_d} f_\tau(u) g_\tau(v-u) e^{-i \tau \beta(u,v)} \, dm_{V_d}(u) = (f_\tau *_\tau g_\tau)(v).
\end{eqnarray*}
This shows that $(f * g)_\tau = f_\tau *_\tau g_\tau$.
\end{proof}
Applying Young's inequality we deduce in particular that
\begin{equation}\label{cor_cft}
\|(f * g)_\tau\|_{L^\infty(V_d)} \leq \|f_\tau\|_{L^2(V_d)} \, \|g_\tau\|_{L^2(V_d)}, \quad \textrm{for all $\tau \in \R$ and $f,g \in L^1(N_d) \cap L^2(N_d)$}.
\end{equation}

\begin{definition} A function $q \in L^\infty(V_d)^{K_d}$ is called \emph{$\tau$-spherical} for $\tau \in \R$ if 
\[
\langle \varphi *_\tau \psi, q \rangle = \langle \varphi , q \rangle \, \langle \psi, q \rangle, \quad \textrm{for all $\varphi, \psi \in L^1(V_d)^{K_d}$}.
\]
\end{definition}
Such functions can be used to parametrize the spherical functions for the Gelfand pair $(G_d, K_d)$:
\begin{proposition}\label{TauSphericalSpherical} Given $\tau \in \R$ and a $\tau$-spherical function $q \in L^\infty(V_d)^{K_d}$, the function $\omega: G_d \to \C$ given by
\begin{equation}
\omega(k, (t,v)) := e^{i\tau t}q(v).
\end{equation}
is a bounded $K_d$-spherical function. Moreover, if $f \in L^1(N_d)^{K_d}$, then
\begin{equation}\label{FTComesBack}
\langle \iota_d(f), \omega \rangle = \langle f_\tau, q\rangle.
\end{equation}
\end{proposition}
\begin{proof} We first prove the second statement: For $f \in L^1(N_d)^{K_d}$ we have
\[
\langle \iota_d(f), \omega \rangle = \int_{N_d} f(t,v) e^{-i\tau t} \overline{q(v)} \, dm_{N_d}(t,v) = \int_{V_d} \int_{\bR} f(t,v) e^{-i\tau t} \, dm_{\bR}(t)\, \overline{q(v)}\, dm_{V_d}(v) = \langle f_\tau, q\rangle.
\]
This establishes \eqref{FTComesBack}, and for $f_1, f_2 \in L^1(N_d)^{K_d}$ we obtain
\begin{eqnarray*}
\langle \iota_d(f_1) \ast \iota_d(f_2), \omega\rangle &=& \langle \iota_d(f_1 \ast f_2), \omega\rangle \quad = \quad \langle (f_1\ast f_2)_\tau, q\rangle \quad = \quad \langle (f_1)_\tau \ast_\tau (f_2)_\tau, q\rangle\\ &=& \langle (f_1)_\tau, q \rangle \langle (f_2)_\tau, q \rangle \quad =\quad \langle \iota_d(f_1), \omega\rangle\langle \iota_d(f_2), \omega\rangle.
\end{eqnarray*}
Since $\iota_d:  L^1(N_d)^{K_d} \to  L^1(G_d, K_d)$ is surjective, this proves that $\omega$ is $K_d$-spherical.
\end{proof}
\begin{remark}[Classification of $(G_d, K_d)$-spherical functions] 
We state without proof the classification of positive-definite $(G_d, K_d)$-spherical functions, which can be found e.g.\ in \cite{Th}, see in particular Proposition 3.2.3 and the remarks after Proposition 3.2.5. It turns out that all bounded spherical function are positive-definite and that they all arise from $\tau$-spherical functions via the construction in Proposition \ref{TauSphericalSpherical}. 

Recall that the \emph{Laguerre polynomial $L_k$ of degree $k$ and type $0$} is defined by 
\[
L_k(t) = e^{-t}\Big(\frac{d}{dt}\Big)^k(e^{t}t^k), \quad \textrm{for $k \in \bN$},
\]
and that the \emph{zeroth Bessel function} $J_o$ is defines as 
\[
J_o(r) = \frac{1}{\pi} \int_0^\pi e^{ir\cos \theta} \, d\theta, \quad \textrm{for $r \geq 0$}.
\]
For $\tau \in \R \setminus \{0\}$, there are countably many $\tau$-spherical functions $\{q_{\tau, \alpha} \mid \alpha \in \bN^d\}$, which can be parametrized by $\bN^d$ and are given by the real-valued functions
\[
q_{\tau, \alpha}(v) = e^{-|\tau||v|^2/4} \cdot \prod_{j=1}^d L_{\alpha_j}(|\tau||v_j|^2/2).
\]
For $\tau = 0$ there are uncountably many $\tau$-spherical function $\{q_{0, \kappa} \mid \kappa \in \R_{\geq 0}^d\}$, which can be parametrized by $\R_{\geq 0}^d$ and are given by
\[
q_{0, \kappa}(v) = \prod_{j=1}^d J_0(\kappa_j |v_j|).
\]
If we denote by $\omega_{\tau, \alpha}$, respectively $\omega_{0, \kappa}$ the $K_d$-spherical functions on $G_d$ corresponding to $q_{\tau, \alpha}$ and $q_{0, \kappa}$ respectively, then we have
\begin{equation}
\mathcal S^+(G_d, K_d) = \mathcal S_b(G_d, K_d) = \{\omega_{\tau,\alpha} \mid \tau \in \R \setminus \{0\}, \alpha \in \bN^d\} \sqcup \{\omega_{0, \kappa} \mid \kappa \in \R_{\geq 0}^d\}.
\end{equation}
The functions $q_{\tau, \alpha}$ are matrix coefficients of Schr\"odinger representations with central character $e^{i \tau}$ and $K_d$ acting by $e^{i\langle \alpha, \cdot \rangle}$ and the functions $q_{0, \kappa}$ correspond to $K_d$-orbits of characters as discussed in the virtually abelian case above. Note that, by the same computation as in the virtually abelian case, for $F \in C_c(G_d, K_d)$, $k \in K_d$ and $\sigma \in V_d$ we have
\begin{equation}\label{BesselExplicit}
\widehat{F}(\omega_{0, |\sigma|}) = \int_{N_d} F(k, (t,v)) e^{-i \langle \sigma, v \rangle} dm_{N_d}(t,v).
\end{equation}
\end{remark}

\subsection{Model sets in minimal Heisenberg motion groups}
We now describe the setting that we will consider throughout the rest of this section. From now on we fix a pair of parameters $\bd = (d_1, d_2)$ and set 
\[
G := G_{d_1} = K_{d_1} \ltimes N_{d_1} = K_{d_1} \ltimes (\R \oplus_{\beta_{d_1}} V_{d_1}), \quad K:= K_{d_1} \qand H := N_{d_2} = \R \oplus_{\beta_{d_2}} V_{d_2}.
\]
Later we will also consider the group
\[
\widetilde{H} := K_{d_2} \ltimes H
\]
We then have 
\[
G \times H = K_{d_1} \ltimes (\R^2 \oplus_{\beta_{\bd}} V_{\bd}) = K_{d_1} \ltimes (N_{d_1} \times N_{d_2})
\]
where $V_{\bd} = V_{d_1} \oplus V_{d_2}$ and $\beta_{\bd}: V_{\bd} \to \R^2$ is the cocycle given by 
\[
\beta_{\bd}((u_1, u_2), (v_1, v_2)) = \begin{pmatrix} \beta_{d_1}(u_1, v_1)\\ \beta_{d_2}(u_2, v_2) \end{pmatrix} \quad (u_1, v_1 \in V_{d_1}, u_2, v_2 \in V_{d_2}).
\]
\begin{remark}[Cut-and-project schemes for Heisenberg motion groups]\label{HeisenbergCUP} We choose lattices $\Delta < V_{\bd}$ and $\Xi < \R^2$ such that $\Delta$ projects densely and injectively onto $V_{d_1}$ and $V_{d_2}$, $\Xi$ projects densely and injectively onto both coordinates and such that $\beta_{\bd}(\Delta, \Delta) \subset \Xi$. We then obtain a lattice $\Gamma_o = \Xi \oplus_{\beta_{\bd}} \Delta$ in $N_{\bd}$, and hence also a lattice
\[
\Gamma := \{((e, (\xi_1, \delta_1), (\xi_2, \delta_2)) \in G \times H \mid (\xi_1, \xi_2) \in \Xi, (\delta_1, \delta_2) \in \Delta\},
\]
in $G \times H$, which we can further extend to a lattice
\begin{equation}\label{GammaTilde}
\widetilde{\Gamma} := \{((e, \gamma_1), (e, \gamma_2)) \in G \times \widetilde{H} \mid (\gamma_1, \gamma_2) \in \Gamma_o\} < G \times \widetilde{H}.
\end{equation}
in $G \times \widetilde{H}$. Then by construction $(G, H, \Gamma)$ is a cut-and-project scheme for the minimal Heisenberg motion group $G$. \end{remark}
\begin{remark}[Model sets in the Heisenberg motion group] Let $(G, H, \Gamma)$ be the cut-and-project scheme from Remark \ref{HeisenbergCUP}. Since $\Gamma$ is countable, we can choose parameters $a_j, b_j \in \R$ such that the boundary of
\[
W := \{(t, (z_1, \dots, z_d)) \in H \mid t \in [a_0, b_0], |z_j| \in [a_j, b_j] \}
\]
does not intersect the projection of $\Gamma$ to $H$. We fix such parameters once and for all. Note that $W$ splits as a product
\[
W = I \otimes W_o, \quad \text{where } I = [a_0, b_0] \text{ and } W_o = \{(z_1, \dots, z_d) \in V_{d_2} \mid |z_j| \in [a_j, b_j]\}.
\]
We then obtain a uniform regular model set
\[
\widetilde{\Lambda} := \widetilde{\Lambda}(G, H, \Gamma, W) = {\rm proj}_G(\Gamma \cap (G \times W)),
\]
and the associated uniform regular model set in the Heisenberg group $K\backslash G = N_{d_1}$ is then given by
\[
\Lambda  = {}_Kp(\widetilde{\Lambda}) = {\rm proj}_{N_{d_1}}(\Gamma_o \cap (N_{d_1} \times W)).
\]
\end{remark}
For the remainder of this section we assume that $\widetilde{\Lambda}$ and $\Lambda$ are defined as above. We then denote by $\nu$ the unique $G$-invariant measure on the hull of $\Lambda$ and by $\eta$ the associated auto-correlation measure. Our goal is to determine the corresponding diffraction measure $\widehat{\eta}$. From Proposition \ref{DiffCoeffConvenient} we know that
\[
\widehat{\eta}(|\widehat{f}|^2) =  \sum_{(e, \gamma_1),\gamma_2) \in \Gamma} (f \ast f^*)(\gamma_1) ({\bf 1}_W \ast {\bf 1}_{W^{-1}})(\gamma_2) \quad  (f \in C_c(G, K) = C_c(G_{d_1}, K_{d_1})).
\]
If we denote by $\widetilde{\Gamma} < G \times \widetilde{H}$ the lattice from \eqref{GammaTilde}, then we can rewrite this as
\begin{equation}\label{HeisenbergStart}
\widehat{\eta}((|\widehat{f}|^2) = \delta_{\widetilde{\Gamma}}\left((f \otimes {\bf 1}_{\{e\} \times W})\ast (f \otimes {\bf 1}_{\{e\} \times W})^* \right).
\end{equation}
This formula is the starting point of our investigation.

\subsection{Regularization}
A technical difficulty in manipulating the right-hand side of \eqref{HeisenbergStart} arises from the fact that $f$ is not smooth, and that ${\bf 1}_W$ is not even continuous. To circumvent this problem we argue as follows. 

The key observation is that not only $(G, K_{d_1})$, but also $(G \times \widetilde{H}, K_{d_1} \times K_{d_2})$ is a Gelfand pair, and that $\delta_{\widetilde{\Gamma}}$ is a positive-definite Radon measure on $G \times \widetilde{H}$. It thus follows from the absolute case of the Godement-Plancherel theorem (Corollary \ref{GodementPlancherel}) that the measure $\delta_{\widetilde{\Gamma}}$ admits a spherical Fourier transform $\widehat{\delta}_{\widetilde{\Gamma}}$ with respect to the Gelfand pair $(G \times \widetilde{H}, K_{d_1} \times K_{d_2})$. This measure satisfies
\begin{equation}\label{FTDeltaGamma}
\widehat{\delta}_{\widetilde{\Gamma}}(|\widehat{F}|^2) = \delta_{\widetilde{\Gamma}}(F \ast F^*)
\end{equation}
for all bounded measurable bi-$K_{d_1} \times K_{d_2}$-invariant functions $F: G \times \widetilde{H} \to \C$ with compact support, but it is already uniquely determined by the fact that it satisfies \eqref{FTDeltaGamma} for all \emph{smooth} functions $F \in C^\infty_c(G \times \widetilde{H},  K_{d_1} \times K_{d_2})$. From the former property and \eqref{HeisenbergStart} we may deduce that
\[
\widehat{\eta}(|\widehat{f}|^2) = \widehat{\delta}_{\widetilde{\Gamma}} (|\widehat f|^2 \otimes |\widehat{{\bf 1}}_{\{e\} \times W}|^2),
\]
and using the latter property we deduce:
\begin{lemma}[Regularity lemma]\label{RegularityLemma}
Let $m \in M^+(\mathcal S^+(G \times \widetilde{H},  K_{d_1} \times K_{d_2}))$ and assume that for all smooth functions $f_1 \in  C^\infty_c(N_{d_1})^{K_{d_1}} $ and $f_2 \in C^\infty_c(N_{d_2})^{K_{d_2}}$ we have
\[
\sum_{(\delta_1, \delta_2) \in \Delta} \sum_{(\xi_1, \xi_2) \in \Xi} (f_1 \ast f_1^*)(\xi_1, \delta_1) (f_2 \ast f_2^*)(\xi_2, \delta_2) = m(|\widehat{\iota_{d_1}^{-1}({f_1}})|^2 \otimes |\widehat{\iota_{d_2}^{-1}(f_2)}|^2).
\]
Then $m =  \widehat{\delta}_{\widetilde{\Gamma}}$, and thus for all $\psi \in C_c(\mathcal S^+(G, K_{d_1}))$ we have
\[\pushQED{\qed}
\widehat{\eta}(\psi) = m(\psi \otimes (|\widehat{{\bf 1}}_{\{e\} \times W}|^2)).\qedhere \popQED
\]
\end{lemma}

\subsection{The Poisson summation formula and the horizontal contribution}
From now on we consider functions $f_1 \in  C^\infty_c(N_{d_1})^{K_{d_1}} $ and $f_2 \in C^\infty_c(N_{d_2})^{K_{d_2}}$. By Lemma \ref{RegularityLemma}, in order to determine the auto-correlation measure we have to express the sum
\[
\Xi(f_1, f_2) := \sum_{(\delta_1, \delta_2) \in \Delta} \sum_{(\xi_1, \xi_2) \in \Xi} (f_1 \ast f_1^*)(\xi_1, \delta_1) (f_2 \ast f_2^*)(\xi_2, \delta_2)
\]
in terms of the spherical Fourier transforms of the functions $F_1 := {\iota_{d_1}^{-1}(f_1)}$ and $F_2 := {\iota_{d_2}^{-1}(f_2)}$. 

If we formally apply the Poisson summation formula in the $\xi$-variables and then apply Lemma \ref{lemma_cft} we obtain
\begin{eqnarray*}
\Xi(f_1, f_2) &=& \sum_{(\delta_1, \delta_2) \in \Delta} \sum_{(\xi_1, \xi_2) \in \Xi} (f_1 \ast f_1^*)(\xi_1, \delta_1) (f_2 \ast f_2^*)(\xi_2, \delta_2)\\
&=& \sum_{(\delta_1, \delta_2) \in \Delta} \sum_{(\tau_1, \tau_2) \in \Xi^\perp} (f_1 \ast f_1^*))_{\tau_1}(\delta_1) (f_2 \ast f_2^*)_{\tau_2}(\delta_2)\\
&=& \sum_{(\tau_1, \tau_2) \in \Xi^\perp}  \sum_{(\delta_1, \delta_2) \in \Delta} ((f_1)_{\tau_1} \ast_{\tau_1} (f_1)^*_{\tau_1})(\delta_1)((f_2)_{\tau_2} \ast_{\tau_2} (f_2)^*_{\tau_2})(\delta_2).
\end{eqnarray*}
Here the sum over $\Delta$ is actually finite, hence the final rearrangement is legitimate. To justify the application of the Poisson summation formula, we need to ensure enough decay of the function 
\[
\R^2 \to \C, \quad \tau \mapsto \big((f_1 * f_1^*) \otimes (f_2 * f_2^*)\big)_\tau(\delta_1,\delta_2), \quad \textrm{for $\tau \in \bR^2$},
\]
for a fixed $(\delta_1,\delta_2) \in \Delta$. By \eqref{cor_cft}, we know that the absolute value of 
this function is bounded from above by the function 
\[
\tau \mapsto \|(f_1)_{\tau_1}\|^2_{L^2(V_{d_1})} \, \|(f_2)_{\tau_2}\|^2_{L^2(V_{d_2})}.
\]
To justify our formal computation  we thus need to ensure that this majorant is summable over $\Xi^\perp$. This, however, follows from the smoothness assumptions on $f_1$ and $f_2$, which ensure that their central Fourier transforms decay superpolynomially fast.

Now recall from Lemma \ref{CPPropertyAbelian} that $\Xi^\perp \subset \widehat{\R^2}$ projects injectively onto both coordinates, hence $(0,0)$ is the only point in $\Xi^\perp$ which has a $0$ coordinate. We may thus split $\Xi(f_1, f_2)$ into a sum of a \emph{horizontal part} (corresponding to $(\tau_1, \tau_2) = (0,0)$)
\begin{equation}\label{HorizontalPart}
\Xi_{\rm hor}(f_1, f_2) =   \sum_{(\delta_1, \delta_2) \in \Delta} ((f_1)_{0} \ast_{0} (f_1)^*_0)(\delta_1)((f_2)_{0} \ast_{0} (f_2)^*_0)(\delta_2),
\end{equation}
and a \emph{vertical part}
\begin{equation}\label{VerticalPart}
\Xi_{\rm ver}(f_1, f_2)  = \underset{\tau_1 \neq 0 \neq \tau_2}{\sum_{(\tau_1, \tau_2) \in \Xi^\perp}}  \sum_{(\delta_1, \delta_2) \in \Delta} ((f_1)_{\tau_1} \ast_{\tau_1} (f_1)^*_{\tau_1})(\delta_1)((f_2)_{\tau_2} \ast_{\tau_2} (f_2)^*_{\tau_2})(\delta_2).\end{equation}
The computation of the horizontal part is exactly as in the virtually abelian case: If for $j\in \{1,2\}$ we abbreviate $g_j := (f_j)_0$, then using that $\ast_0$ is just the usual (untwisted) convolution then the Poisson summation formula (which applies in view of the regularity assumptions of the $f_j$) yields
\[
\Xi_{\rm hor}(f_1, f_2) = \sum_{(\sigma_1, \sigma_2) \in \Delta^\perp} |\widehat{g_1}(\sigma_1)|^2 |\widehat{g_2}(\sigma_2)|^2.
\]
Now for $\sigma_j \in V_{d_j}$ we have by definition of $g_j$ and $F_j$ and \eqref{BesselExplicit}
\begin{eqnarray*}
|\widehat{g_j}(\sigma_j)| &=& \int_{V_{d_j}} g_j(v_j) e^{-i \langle \sigma_j, v_j\rangle} dm_{V_{d_j}}(v_j)\\
&=& \int_{N_{d_j}} f_j(t_j, v_j) e^{-i \langle \sigma_j, v_j \rangle}  dm_{N_{d_j}}(t_j, v_j)\\
&=&  \int_{N_d} F_j(e, (t,v)) e^{-i \langle \sigma_j, v_j \rangle} dm_{N_d}(t,v)\\
&=& \widehat{F}_j(\omega_{0, |\sigma_j|}).
\end{eqnarray*}
We have thus established:
\begin{corollary}[Computation of the horizontal contribution]\label{Horizontal}
For $f_1 \in  C_c^\infty(N_{d_1})^{K_1}$ and $f_2 \in  C_c^\infty(N_{d_2})^{K_2}$ we have
\[\pushQED{\qed}
\Xi_{\rm hor}(f_1, f_2)  = 
 \sum_{(\sigma_1, \sigma_2) \in \Delta^\perp}  |\widehat{F}_1(\omega_{0, |\sigma_1|})|^2  |\widehat{F}_2(\omega_{0, |\sigma_2|})|^2.\qedhere\popQED
\]
\end{corollary}

\subsection{Laguerre polynomials and the vertical contribution}
The computation of the vertical part requires entirely new arguments based on properties of Laguerre polynomials. We have
\[
\Xi_{\rm ver}(f_1, f_2)  =  \underset{\tau_1 \neq 0 \neq \tau_2}{\sum_{(\tau_1, \tau_2) \in \Xi^\perp}} S_{\tau_1, \tau_2}, \quad \text{where }S_{\tau_1, \tau_2} := \sum_{(\delta_1, \delta_2) \in \Delta} ((f_1)_{\tau_1} \ast_{\tau_1} (f_1)^*_{\tau_1})(\delta_1)((f_2)_{\tau_2} \ast_{\tau_2} (f_2)^*_{\tau_2})(\delta_2),
\]
and we are going to first consider the sum $S_{\tau_1, \tau_2}$ for a fixed pair $(\tau_1, \tau_2) \in \Xi^\perp$ with $\tau_1\neq 0 \neq \tau_2$. We are going to use the following properties of the functions $q_{\tau, \alpha}$.
\begin{proposition}[Properties of $\tau$-spherical functions]\label{tauProperties} Let $\tau \in \R \setminus \{0\}$.
\begin{enumerate}[(i)]
\item The functions $(q_{\tau, \alpha})_{\alpha \in \mathbb N^d}$ form an orthogonal basis for $L^2(V_d)^{K_d}$ with
\[
\|q_{\tau, \alpha}\|_{L^2(V_d)} = \left({2\pi}{|\tau|^{-1}}\right)^{d/2}.
\]
\item For $\alpha, \beta \in  \mathbb N^d$ with $\alpha \neq \beta$ we have
\[
q_{\tau, \alpha} \ast_\tau q_{\tau, \alpha} = q_{\tau, \alpha} \qand q_{\tau, \alpha} \ast_\tau q_{\tau, \beta} = 0. 
\]
\item For all $\alpha \in \bN^d$ we have
\[
\|q_{\tau, \alpha}\|_\infty \leq \left({2\pi}{|\tau|^{-1}}\right)^d.
\]
\item For all $\varphi, \psi \in L^1(V_d)^{K_d} \cap L^2(V_d)^{K_d}$ we have
\[
\phi \ast_\tau \psi = \frac{|\tau|^d}{(2\pi)^d} \sum_{\alpha \in \mathbb N^d} \langle \varphi, q_{\tau, \alpha} \rangle  \langle \psi, q_{\tau, \alpha} \rangle q_{\tau, \alpha},
\]
where the convergence is absolute in $L^\infty$-norms, and thus uniform in $C_b(V_d)^{K_d}$.
\end{enumerate}
\end{proposition}
\begin{proof} (i) See e.g.\ \cite[Prop. 1.4.1]{Th}. (ii) From $\tau$-sphericity we deduce that for $\alpha, \beta, \gamma \in \bN^d$ with $\alpha\neq \beta$ we have
\[
\langle q_{\tau, \alpha} \ast q_{\tau\beta}, q_{\tau, \gamma} \rangle =\langle q_{\tau, \alpha} , q_{\tau, \gamma}\rangle \langle  q_{\tau\beta}, q_{\tau, \gamma} \rangle=0,
\]
since the first factor vanishes if $\gamma \neq \alpha$ and the second factor vanishes if $\gamma = \alpha$. Similarly,
\[
\langle q_{\tau, \alpha} \ast_\tau q_{\tau, \alpha}, q_{\tau, \beta}\rangle = \langle q_{\tau, \alpha}, q_{\tau, \beta}\rangle^2 = 0.
\]
It thus follows from (i) that $q_{\tau, \alpha} \ast_\tau q_{\tau, \alpha} = \lambda q_{\tau, \alpha}$ for some $\lambda \in \C$, and $\tau$-sphericity yields
\[
\lambda = \left({2\pi}{|\tau|^{-1}}\right)^{-d}\langle \lambda q_{\tau, \alpha}, q_{\tau, \alpha} \rangle = \left({2\pi}{|\tau|^{-1}}\right)^{-d}\langle q_{\tau, \alpha} \ast_\tau q_{\tau, \alpha}, q_{\tau, \alpha} \rangle =  \left({2\pi}{|\tau|^{-1}}\right)^{-d}\langle q_{\tau, \alpha}, q_{\tau, \alpha}\rangle^2 =1.
\]
(iii) follows from (i), (ii) and \eqref{cor_cft}, since 
\[\|q_{\tau, \alpha}\|_\infty = \|q_{\tau, \alpha} \ast_\tau q_{\tau, \alpha}\|_\infty \leq \|q_{\tau, \alpha}\|^2_{L^2(V_d)}\]
(iv) By (i) and (ii) we have
\begin{eqnarray*}
\phi \ast_\tau \psi &=& \left(\sum_{\alpha \in\bN^d} (|\tau|(2\pi)^{-1})^{d/2} \langle \varphi, q_{\tau, \alpha} \rangle q_{\tau, \alpha}\right) \ast_\tau \left(\sum_{\beta \in\bN^d} (|\tau|(2\pi)^{-1})^{d/2} \langle \psi, q_{\tau, \beta} \rangle q_{\tau, \beta}\right)\\
&=& \frac{|\tau|^d}{(2\pi)^d} \sum_{\alpha \in \mathbb N^d} \langle \varphi, q_{\tau, \alpha} \rangle  \langle \psi, q_{\tau, \alpha} \rangle q_{\tau, \alpha}.
\end{eqnarray*}
The proof that the convergence is absolute in $L^\infty$-norm (and hence uniform) can be seen as follows. By the Cauchy--Schwartz and Bessel inequalities, the 
function $\alpha \mapsto \langle \varphi, q_{\tau,\alpha} \rangle \langle \psi, q_{\tau,\alpha} \rangle$ belongs to $\ell^1(\bN^d)$.
Since $\alpha \mapsto \|q_{\tau,\alpha}\|_{L^\infty(V)}$ is bounded by (ii), we see that the sum of the 
$L^\infty$-norms of the terms in the sum above converge, hence the absolute convergence follows from the Weierstra\ss\ $M$-test.
\end{proof}
Applying (iv) to the functions $f_j, f_j^* \in C_c^\infty(N_{d_j})^{K_j} \subset L^1(N_{d_j})^{K_{d_j}} \cap L^2(N_{d_j})^{K_{d_j}}$ we obtain
\[
((f_j)_{\tau_j} \ast_{\tau_j} (f_j)^*_{\tau_j})(\delta_j) = \frac{|\tau_j|^{d_j}}{(2\pi)^{d_j}} \sum_{\alpha \in \mathbb N^{d_j}} \langle (f_j)_{\tau_j}, q_{\tau_j, \alpha} \rangle  \langle (f_j^*)_{\tau_j}, q_{\tau_j, \alpha} \rangle q_{\tau_j, \alpha}(\delta_j).
\]
Recall that $F_j = {\iota_{d_j}^{-1}(f_j)} \in C_c^\infty(K_{d_j} \ltimes N_{d_j}, K_{d_j})$. Using \eqref{FTComesBack} we obtain
\[
\langle (f_j)_{\tau_j}, q_{\tau_j, \alpha} \rangle  \langle (f_j^*)_{\tau_j}, q_{\tau_j, \alpha} \rangle =  \langle F_j, \omega_{\tau_j, \alpha} \rangle  \langle F_j^*, \omega_{\tau_j, \alpha} \rangle = |\widehat{F}_j(\omega_{\tau_j, \alpha})|^2,
\]
and hence we find
\[
S_{\tau_1, \tau_2} = \frac{|\tau_1|^{d_1}|\tau_2|^{d_2}}{(2\pi)^{d_1+d_2}}\sum_{(\delta_1, \delta_2) \in \Delta}\sum_{(\alpha, \beta) \in \bN^{d_1+d_2}} |\widehat{F}_1(\omega_{\tau_1, \alpha})|^2 |\widehat{F}_2(\omega_{\tau_2, \beta})|^2q_{\tau_1, \alpha}(\delta_1)q_{\tau_2, \beta}(\delta_2).
\]
Using properties of Laguerre polynomials one can show:
\begin{lemma}[Absolute convergence]\label{AbsoluteConvergence} For all $\tau_1 \neq 0 \neq \tau_2$ the expression
\[
\sum_{(\delta_1, \delta_2) \in \Delta}\sum_{(\alpha, \beta) \in \bN^{d_1+d_2}} |\widehat{F}_1(\omega_{\tau_1, \alpha})|^2 |\widehat{F}_2(\omega_{\tau_2, \beta})|^2q_{\tau_1, \alpha}(\delta_1)q_{\tau_2, \beta}(\delta_2)
\]
converges absolutely.
\end{lemma}
We defer the proof to Appendix \ref{AppendixLaguerre}, but emphasize that smoothness of $f_1$ and $f_2$ enters crucially. Using absolute convergence we may now freely reorder the sums inside $S_{\tau_1, \tau_2}$ for every fixed $(\tau_1, \tau_2)$. In particular, if we define a function $\sigma^\Delta_{\tau_1, \tau_2}: \bN^{d_1} \times \bN^{d_2} \to \C$ by
\begin{equation}\label{SigmaFunction}
\sigma^\Delta_{\tau_1, \tau_2}(\alpha,\beta) := \frac{|\tau_1|^{d_1}|\tau_2|^{d_2}}{(2\pi)^{d_1+d_2}}  \, \sum_{(\delta_1,\delta_2) \in \Delta} q_{\tau_1,\alpha}(\delta_1) q_{\tau_2,\beta}(\delta_2),
\end{equation}
then we obtain:
\begin{corollary}[Computation of the vertical contribution]\label{Vertical}
For $f_1 \in  C_c^\infty(N_{d_1})^{K_1}$ and $f_2 \in  C_c^\infty(N_{d_2})^{K_2}$ we have
\[\pushQED{\qed}
\Xi_{\rm ver}(f_1, f_2)  = \underset{\tau_1 \neq 0 \neq \tau_2}{\sum_{(\tau_1, \tau_2) \in \Xi^\perp}}\sum_{(\alpha, \beta) \in \bN^{d_1+d_2}} \sigma^\Delta_{\tau_1, \tau_2}(\alpha,\beta)  |\widehat{F}_1(\omega_{\tau_1, \alpha})|^2 |\widehat{F}_2(\omega_{\tau_2, \beta})|^2.\qedhere\popQED
\]
\end{corollary}



\subsection{The polyradial diffraction formula}
Combining the vertical and the horizontal contribution we finally obtain with Lemma \ref{RegularityLemma} the following formula:
\begin{theorem}[Diffraction formula for the minimal Heisenberg motion group]\label{DiffMain} The diffraction measure $\widehat{\eta}$ of the regular model set $\Lambda = \Lambda(G_{d_1}, H, \Gamma, W = I \times W_o)$ is given by the formula
\[
\widehat{\eta} =  \sum_{(\sigma_1, \sigma_2) \in \Delta^\perp}  |m_\R(I)|^2 |\widehat{\bf 1}_{W_o}(\sigma_2)|^2 \delta_{\omega_{0, |\sigma_1|}} + \underset{\tau_1 \neq 0 \neq \tau_2}{\sum_{(\tau_1, \tau_2) \in \Xi^\perp}}\sum_{(\alpha, \beta) \in \bN^{d_1+d_2}} \sigma^\Delta_{\tau_1, \tau_2}(\alpha,\beta)  |\widehat{\bf 1}_I(\tau_2)|^2 |\langle {\bf 1}_{W_o}, q_{\tau_2, \beta} \rangle|^2 \delta_{\omega_{\tau_1, \alpha}},
\]
where $\sigma^\Delta: \Xi^\perp \setminus\{(0, 0)\} \to \C$ is given by \eqref{SigmaFunction}.
\end{theorem}
\begin{proof} By Lemma \ref{RegularityLemma}, Corollary \ref{Horizontal} and Corollary \ref{Vertical} we have
\[
\widehat{\delta}_{\widetilde{\Gamma}} = \sum_{(\sigma_1, \sigma_2) \in \Delta^\perp} \delta_{\omega_{0, |\sigma_1|}} \otimes \delta_{\omega_{0, |\sigma_2|}} + \underset{\tau_1 \neq 0 \neq \tau_2}{\sum_{(\tau_1, \tau_2) \in \Xi^\perp}}\sum_{(\alpha, \beta) \in \bN^{d_1+d_2}} \sigma^\Delta_{\tau_1, \tau_2}(\alpha,\beta) \cdot \delta_{\omega_{\tau_1, \alpha}} \otimes \delta_{\omega_{\tau_2,\beta}}.
\]
On the other hand, by Lemma \ref{RegularityLemma} we also have
\[\widehat{\eta}(\psi) = m(\psi \otimes (|\widehat{{\bf 1}}_{\{e\} \times W}|^2)), \quad (\psi \in C_c(\mathcal S^+(G, K_{d_1}))),
\]
and since $W = I \times W_o$ we see from \eqref{BesselExplicit} and Proposition \ref{TauSphericalSpherical} that
\[
|\widehat{{\bf 1}}_{\{e\} \times W}(\omega_{0, |\sigma_2|})|^2 = |m_\R(I)|^2 |\widehat{\bf 1}_{W_o}(\sigma_2)|^2 \qand
|\widehat{{\bf 1}}_{\{e\} \times W}(\omega_{\tau_2, \beta})|^2 =  |\widehat{\bf 1}_I(\tau_2)|^2 |\langle {\bf 1}_{W_o}, q_{\tau_2, \beta} \rangle|^2.\]
The theorem follows.
\end{proof}

\section{Semisimple examples: ${\rm SL}_2(\R)$}\label{SecSemisimple}

\subsection{The auto-correlation distribution of a weighted model set in the hyperbolic plane}
As explained in \cite{BHP2} the auto-correlation measure of a weighted model set in the hyperbolic plane can be re-interpreted as an evenly positive-definite distribution on the real line. We briefly recall the relevant results and notations. As in \cite{BHP2} we define elements of $G := {\rm SL}_2(\R)$ by
\[
k_\theta 
:=
\left(
\begin{matrix}
\cos 2\pi \theta& \sin 2\pi \theta \\ 
-\sin 2\pi \theta & \cos 2\pi \theta
\end{matrix}
\right),
\quad
a_t 
:=
\left(
\begin{matrix}
e^{t/2} & 0 \\ 
0 & e^{-t/2}
\end{matrix}
\right)
\qand 
n_u 
:=
\left(
\begin{matrix}
1 & u \\ 
0 & 1
\end{matrix}
\right).
\]
and denote by $K$, $A$ and $N$ the respective subgroups of $G$ consisting of these matrices. Multiplication induces a diffeomorphism $A \times N \times K \to G$ and thus every $g \in G$ can be written uniquely as
\begin{equation}
g = a_{t} n_{u} k_{\theta}.
\end{equation}
If $f \in L^1(G)$ and $F(t, u, \theta) := f(a_tn_uk_\theta)$, then we will normalize Haar measure on $G$ such that
\[
\int_G f(g) \, dm_G(g) = \int_{[0,1)} \int_{\bR} \int_{\bR} F(t,u,\theta) \, dt \, du \, d\theta.
\]
We then identify
\[
K\backslash G \to \bH^2, \quad Kg \mapsto g^{-1}.i,
\]
where the action of $G$ on $\bH^2$ is by fractional linear transformations. The auto-correlation measure of a weighted regular model set $\Lambda$ in the hyperbolic plane $\bH^2$ then gets identified with a Radon measure $\eta$ on $K\backslash G/K$. Now if we denote by $C_c^\infty(\R)_{\mathrm ev} \subset C_c^\infty(\R)$ the subspace of even functions, then by \cite[Lemma 5.2]{BHP2} (or \cite[Theorems V.2.2 and V.2.3]{Lang}) the Harish transform defines an isomorphism of $*$-algebras 
\[
\bH: C_c^\infty(G, K) \to C_c^\infty(\R)_{\mathrm ev}, \quad (\bH f)(t) = e^{t/2} \int_{\bR} f(a_tn_u) \, du.
\]
By \cite[Prop.\ 5.7]{BHP2} the map
\[
\xi:  C_c^\infty(\R)_{\mathrm ev} \to \C, \quad f \mapsto \eta(\bH^{-1}f)
\]
is an \emph{evenly positive-definite distribution}, i.e.\ a continuous linear functional on $C_c^\infty(\R)_{\mathrm ev}$ such that $\xi(\phi \ast \phi^*) \geq 0$ for all $\phi \in C_c^\infty(\R)_{\mathrm ev}$. Since it determines the auto-correlation measure of $\Lambda$, it is called the \emph{auto-correlation distribution} of $\Lambda$. In view of exponential volume growth of the hyperbolic plane, this distribution is not tempered.

While tempered distribution can be studied via their Fourier transforms, for non-tempered distributions one has to consider a complex version of the Fourier transform known as the \emph{Mellin transform}. Define the Paley-Wiener space $\mathrm{PW}(\C)$ as the space of entire functions $f: \C \to \C$ such that for every $N \in \bN$ there exist constants $C_1, C_2\geq1$ such that
\[
f(\sigma + it) < C_1 \cdot \frac{C_2^{|\sigma|}}{(1+|t|)^N}.
\] 
We denote by $\mathrm{PW}(C)_{\mathrm{ev}} \subset \mathrm{PW}(\C)$ the subspace consisting of functions with $f(-z) = f(z)$. Then the \emph{even Mellin transform} is the isomorphism \cite[Thm.\ V.3.4]{Lang}
\[
\bM: C_c^\infty(\R)_{\mathrm ev} \to \mathrm{PW}(\C)_{\mathrm{ev}}, \quad \bM \phi(z) := \int_{\bR}\phi(t) e^{tz/2} \, dt.
\]
We can thus consider an evenly positive-definite distribution as a linear functional on the even part of a Payley-Wiener space. By a classical result of Gelfand and Vilenkin (generalizing previous work of Krein) such a linear functional actually extends to a Radon measure:
\begin{theorem}[{Gelfand--Vilenkin--Krein, \cite[Thm.\ II.6.5]{GV4}}]\label{GVK} If $\xi: C_c^\infty(\R)_{\mathrm ev} \to \C$ is an evenly positive-definite distribution, then there exists a measure $\mu_\xi \in M^+(\C)$ with $\mathrm{supp}(\mu_\xi) \subset \R \cup i \R$ such that
\begin{equation}\label{MellinDist}\pushQED{\qed}
\xi(\phi) = \mu_\xi(\bM\phi) \quad (\phi \in C_c^\infty(\R)_{\mathrm ev}).\qedhere\popQED
\end{equation}
\end{theorem}
We refer to any measure $\mu_\xi \in M^+(\C)$ which satisfies \eqref{MellinDist} and $\mathrm{supp}(\mu_\xi) \subset \R \cup i \R$ as a \emph{Mellin transform} of the evenly positive-distribution $\xi$. For general evenly positive-definite distributions such a measure is not unique \cite[Sec. II.4]{GV4}. Using the well-known relation between the Mellin transform, the Harish transform and the spherical Fourier transform of the Gelfand pair $(G, K)$ we are going to show:
\begin{theorem}[Pure point diffraction]\label{PurePointDiffraction} Let $\Lambda$ be a uniform regular model set in the hyperbolic plane. Then its auto-correlation distribution $\xi$ admits a Mellin transform $\mu_\xi$ which is a pure point measure.
\end{theorem}

\subsection{Mellin transform vs.\ spherical Fourier transform}
We now explain the proof of Theorem \ref{PurePointDiffraction}. Let $\eta \in M^+(G, K)$ denote the auto-correlation measure of $\Lambda$. By Theorem \ref{AbstractPurePoint} and Proposition \ref{ConcretePurePoint} the diffraction measure $\widehat{\eta} \in M^+(\mathcal S^+(G, K))$ is pure point. We now relate it to the auto-correlation distribution $\xi$ of $\Lambda$.
\begin{remark}[Spherical functions] Denote by $\rho: G \to \R$ the function given by
\[
\rho(a_tn_uk_\theta) = e^{t/2}.
\]
Note that $\rho(a_t) = e^{t/2}$ (corresponding to the half-sum of positive roots) and that $\rho$ is right-$K$-invariant. Integrating complex powers of $\rho$ against the left-$K$-action provides bi-$K$-invariant functions
\[
\omega_z: G \to \C, \quad \omega_z(g) := \int_0^1 \rho(k_\theta g)^{z+1} d\theta,
\]
and it turns out that these are precisely the spherical functions of the Gelfand pair $(G, K)$. Moreover, given $z_1, z_2 \in \C$ we have $\omega_{z_1} = \omega_{z_2}$ iff $z_2 \in \{\pm z_1\}$. We may thus identify the spherical transform of the Gelfand pair $(G,K)$ with the map
\[
\mathbb S: C_c(G, K) \to C(\C)_{\mathrm ev}, \quad \mathbb S(f)(z) := \int_G f(g) \omega_z(g^{-1}) dm_G(x) =  \int_G f(g) \omega_z(g) dm_G(g),
\]
where the final equality follows from the fact that $KgK = Kg^{-1}K$ for all $g \in G$. Finally, the positive definite spherical functions are precisely those of the form $\omega_z$ with $z \in [-1,1] \cup i\R$. For $z \in i\R$ these correspond to spherical principal series, whereas for $z \in [-1,1]$ they correspond to spherical complementary series. We now recall the relation between the Mellin transform, the Harish transform and the spherical transform of the Gelfand pair $(G, K)$ \cite[Thm.\ V.4.5]{Lang}:
\begin{proposition}\label{SMH} For all $f \in C_c^\infty(G, K)$ we have
\[
\mathbb S f = \mathbb M (\mathbb H f) \in  \mathrm{PW}(\C)_{\mathrm{ev}}.
\]
\end{proposition}
\begin{proof} Let $z \in \C$. Using bi-$K$-invariance of $f$ and right-$K$-invariance of $\rho$ we obtain
\begin{eqnarray*}
 \mathbb M (\mathbb H f)(z) &=& \int_\R \left( e^{t/2} \int_\R f(a_tn_u) \, du\right) e^{tz/2}\,dt \quad = \quad \int_\R \int_\R f(a_tn_u) (e^{t/2})^{z+1}\, du \, dt\\
 &=& \int_\R \int_\R \int_0^1 f(a_t n_u k_\theta) \rho(a_t n_u k_\theta) \, d\theta \, du \, dt \quad = \quad \int_G f(g) \rho(g)^{z+1}\, dm_G(g)\\
 &=& \int_G \int_0^1 f(k_{-\theta}g) \rho(g)^{z+1}\, d\theta\, dm_G(g) \quad = \quad   \int_G  f(g) \left(\int_0^1\rho(k_{\theta}g)^{z+1}\, d\theta\right) dm_G(g)\\
 &=& \mathbb S f(z).
\end{eqnarray*}
This proves the proposition.
\end{proof}
If we identify $\mathcal S(G, K)$ with $\C/\{\pm 1\}$ (via $\omega_z \to \{\pm z\}$), then the spherical diffraction measure $\widehat{\eta} \in M^+(\mathcal S^+(G, K))$ corresonds to a pure point Radon measure $\mu$ on $[-1,1] \cup i\R \subset \C$ such that $\mu(A) = \mu(-A)$ and for all $f \in C_c^\infty(G, K)$ we have
\[
\mu(\mathbb S f) = \widehat{\eta}(\widehat{f})= \eta(f).
\]
We thus obtain the following refinement of Theorem \ref{PurePointDiffraction}.
\begin{theorem}\label{MellinPurePoint} The measure $\mu$ is a Mellin transform of the auto-correlation distribution $\xi$. In particular, $\xi$ has a pure point Mellin transform, which is supported on $[-1,1] \cup i\R$.
\end{theorem}
\begin{proof} Let $\phi \in C^\infty_c(\R)_{\mathrm{ev}}$ and $f := \bH^{-1}(\phi) \in C_c^\infty(G, K)$.
In view of Proposition \ref{SMH} we have 
\[
\xi(\phi) = \xi(\bH f) =  \eta(f) = \mu(\bS f) = \mu(\bM(\bH f)) = \mu(\bM \phi).
\]
This proves that $\mu$ is a Mellin transform of $\xi$.
\end{proof}


\end{remark}

\newpage
\appendix

\section{An elementary proof of the Godement--Plancherel theorem}\label{AppendixGodement}
This appendix is devoted to the proof of the Godement-Plancherel theorem in its most general form (Theorem \ref{GodementConvenient}). The proof is by reduction to the spherical Bochner theorem, which is easily accessible from the literature and which we recall in the next subsection. The remainder of the proof is self-contained and inspired by the proof in the abelian case as presented in the book of Berg and Forst \cite{BergFrost}.
\subsection{Reminder of the spherical Bochner theorem}
Recall that $M_b^+(\mathcal S^+(G, K))$ denotes the space of bounded (positive) Radon measures on the locally compact space $\mathcal S^+(G, K)$ of positive-definite spherical functions. Our starting point is the following spherical version of the classical Bochner theorem \cite[Thm.\ 9.3.4]{Wolf-07}:
\begin{theorem}[Spherical Bochner theorem]\label{Bochner} Let $\varphi \in P(G,K)$. Then there exists a unique $\mu_\varphi \in M_b^+(\mathcal S^+(G, K))$ such that
\[\pushQED{\qed}
\varphi(x) = \int_{\mathcal S^+(G, K)} \overline{\omega(x)} d\mu_\varphi(\omega).\qedhere \popQED
\]
\end{theorem}
\begin{definition}\label{AssociatedMeasure}
For $\varphi \in P(G, K)$ the measure $\mu_\varphi$ from Theorem \ref{Bochner} is called the \emph{associated measure} of $\varphi$.
\end{definition}
\subsection{A convenient reformulation of the Godement-Plancherel theorem}
We now turn to the proof of Theorem \ref{GodementConvenient}. From now on we fix a measure $\mu \in M(G, K)$ which is of positive type relative $K$. We have to show existence and uniqueness of a measure $\widehat{\mu} \in M(\mathcal S^+(G, K))$ satisfying the equivalent conditions (God1)-(God3) from Proposition \ref{GConditions} and to show that it is positive. That $\widehat{\mu}$ uniquely determines $\mu$ is immediate from (God2) and the fact that $\{f\ast f^* \mid f \in C_c(G, K)\}$ spans a dense subspace of $C_c(G, K)$, \cite[Lemma~A.13]{BHP2}. Using bi-$K$-invariance of $\mu$ we can reformulate Conditions (God1)-(God3) as follows:
\begin{lemma}\label{God4} For a right-$K$-invariant measure $\mu \in M(G)$ Conditions (God1)-(God3) from Proposition \ref{GConditions} are equivalent to the following condition.
\begin{enumerate}[(God1)]\setcounter{enumi}{3}
\item For every $f \in C_c(G,K)$ we have $\widehat{f}\in L^2(\mathcal S^+(G, K), \widehat{\mu})$ and for every $f \in C_c(G)$ and $x \in G$ we have
\[
\mu \ast f \ast f^*(x) = \int_{\mathcal S^+(G, K)} |\widehat{f}(\omega)|^2 \overline{\omega(x)} d \widehat{\mu}(\omega).
\]
\end{enumerate}
\end{lemma}
\begin{proof} Since $\omega(e) = 1$, (God2) follows from (God4) by choosing $x := e$. For the converse, assume that (God3) holds, let $f \in C_c(G, K)$, $x \in G$ and define $g := L_x^\sharp f \in C_c(G, K)$ as in Lemma \ref{LeftTranslationSharp}. We recall from \eqref{FTTranslation} that 
\[
\widehat{g}(\omega) = \widehat{f}(\omega) \cdot {\omega(x)}.
\]
Then, using right-$K$-invariance of $\mu$ and applying (God3) we obtain
\begin{eqnarray*}
\mu \ast f \ast f^*(x) &=& \int_G f\ast f^*(y^{-1}x)d\mu(y) \quad = \quad \int_G\int_G f(z)f^*(z^{-1}y^{-1}x)\, dm_G(z) \, d\mu(y)\\
&=& \int_G \int_G f(z) \overline{f(x^{-1}ky^{-1}z)} \, dm_G(z) \, d\mu(yk^{-1})dm_K(k)\\
&=& \int_G \int_G f(z) \overline{L^\sharp_x f(y^{-1}z)} \, dm_G(z) \, d\mu(y)\quad = \quad \mu(f \ast g^*)\\
&=&  \int_{\mathcal S^+(G, K)}  \widehat{f}(\omega) \overline{\widehat{g}(\omega)} \; d\widehat{\mu}(\omega) \quad = \quad \int_{\mathcal S^+(G, K)} |\widehat{f}(\omega)|^2 \overline{\omega(x)} d \widehat{\mu}(\omega),
\end{eqnarray*}
which establishes (God4) and finishes the proof.
\end{proof}
To simplify this condition further, we use the following relation between measures and functions of positive type, which in the abelian case was pointed out in \cite[Prop. 4.4]{BergFrost}.
\begin{lemma}\label{GodementMainLemma} Let $(G, K)$ be a Gelfand pair. If $\mu \in M(G, K)$ is of positive type relative $K$, then for all $f \in C_c(G, K)$ the function $\mu \ast f \ast f^*$ is positive-definite and continuous, hence of positive type.
\end{lemma}
The proof relies on the following slight extension of Axiom (Gel3) of a Gelfand pair.
\begin{lemma}\label{RadonCommutative} Let $(G, K)$ be a Gelfand pair and $\mu, \nu \in M(G, K)$. If at least one of the two measures has compact support, then $\mu \ast \nu$ and $\nu \ast \mu$ converge and satisfy $\mu \ast \nu = \nu \ast \mu$. In particular $\mu \ast f = f \ast \mu$ for all $f \in C_c(G, K)$.
\end{lemma}
\begin{proof} Assume that $\nu \in M(G, K)$ has compact support and let $h \in C_c(G)$. Choose a compact set $C\subset G$ which contains both ${\rm supp}(h) {\rm supp}(\nu)^{-1}$ and ${\rm supp}(\nu)^{-1}{\rm supp}(h)$. Then there exists a measure $\mu_0 \in M_b(G, K)$ which coincides with $\mu$ on $C$, and using commutativity of $M_b(G, K)$ we have
\[
(\mu \ast \nu)(h) = (\mu_0 \ast \nu)(h) = (\nu \ast \mu_0)(h) = (\nu \ast \mu)(h),\]
hence $\mu \ast \nu = \nu \ast \mu$.
\end{proof}
\begin{proof}[Proof of Lemma \ref{GodementMainLemma}] Note first that by Lemma \ref{RadonCommutative} we have $\mu \ast f \ast f^* = f \ast \mu \ast f^*$, hence for all $f \in C_c(G)$ and $x \in G$ we have
\[
\mu \ast f \ast f^*(x) \;=\; f \ast \mu \ast f^*(x) \;=\; \int_G f(z) \int_G f^*(y^{-1}z^{-1}x)d\mu(y)dm_G(z)\;=\; \int_G \int_G f(z)\overline{f(x^{-1}zy)}dm_G(z) d\mu(y).
\]
Now let $f \in C_c(G)$, $\lambda_1, \dots, \lambda_n \in \C$, $g_1, \dots, g_n \in G$. We define $h \in C_c(G, K)$ by 
\[
h(z) := \sum_{i=1}^n \lambda_i \int_K f(x_i k z) dm_K(k).
\]
Using Lemma \ref{RadonCommutative} again we deduce that
\begin{eqnarray*} 
\sum \lambda_i\overline{\lambda_j}(\mu \ast f \ast f^*)(x_ix_j^{-1}) &=& \sum \lambda_i\overline{\lambda_j} \int_G \int_G f(z)\overline{f(x_jx_i^{-1}zy)}dm_G(z) d\mu(y)\\
&=& \sum \lambda_i\overline{\lambda_j} \int_G \int_G \int_K \int_K f(x_ik_1z)\overline{f(x_jk_2zy)}dm_K(k_1)dm_K(k_2)dm_G(z) d\mu(y)\\
&=& \int_G \int_G h(z)\overline{h(zy)}dm_G(z) d\mu(y)\\
&=& h \ast \mu \ast h^*(e) \quad = \quad \mu \ast h \ast h^*(e)\\
&=& \mu(h \ast h^*) \quad \geq \quad 0.
\end{eqnarray*}
This shows that $\mu \ast f \ast f^*$ is positive-definite, and continuity is obvious.
\end{proof}
Combining this with Lemma \ref{God4} and the spherical Bochner theorem (Theorem \ref{Bochner}) we have reached the following convenient reformulation of Conditions (God1)-(God3):
\begin{corollary}\label{God5} If $\mu \in M(G, K)$ is of positive type relative $K$, then a measure $\widehat{\mu} \in M(\mathcal S^+(G, K))$ is a Fourier transform of $\mu$ if and only if the following condition holds:
\begin{enumerate}[(God1)]\setcounter{enumi}{4}
\item For every $f \in C_c(G)$ we have $\widehat{f}\in L^2(\mathcal S^+(G, K), \widehat{\mu})$ and $|\widehat{f}|^2 \widehat{\mu} = \sigma_f$, where $\sigma_f$ denotes the associated measure of $\mu \ast f \ast f^*$ in the sense of Definition \ref{AssociatedMeasure}.\qed
\end{enumerate}
\end{corollary}

\subsection{Proof of the Godement--Plancherel theorem}
We now show that given $\mu \in M(G, K)$ which is of positive type relative $K$ there is a unique (positive) measure satisfying the condition (God5) from Corollary \ref{God5}. For the proof of uniqueness we need the following auxiliary observation:
\begin{lemma}\label{Godement2} For every compact subset $C \subset \mathcal S^+(G, K)$ there exists a positive-definite function $f \in C_c(G,K)$ such that $\widehat{f}(\omega) > 0$ for all $\omega \in C$.
\end{lemma}
\begin{proof} Since $\mathcal F(C_c(G, K))$ is dense in $C_0(\mathcal S_{b}(G, K))$ we find for every $\omega \in \Omega$ some $h_\omega \in C_c(G)$ such that $\widehat{h}_\omega(\omega) \neq 0$. Then $f_\omega := h_\omega \ast h^*_\omega$ is continuous, positive-definite and compactly supported, and
\[
\widehat{f}_\omega(\xi) = |\widehat{h}_\omega(\xi)|^2 \geq 0 \text{ for all }\xi \in \mathcal S_b(G, K) \qand \widehat{f}_\omega(\omega)= |\widehat{h}_\omega(\omega)|^2 > 0.
\]
Consequently, $\widehat{f}_\omega$ is non-negative, and strictly positive on some open neighbourhood $U_\omega$ of $\omega$. Since $C$ is compact there exist $\omega_1, \dots, \omega_n \in C$ such that $U_{\omega_1}, \dots, U_{\omega_n}$ cover $C$, and then $f := f_{\omega_1} + \dots + f_{\omega_n}$ has the desired properties.
\end{proof}
Now let $\psi \in C_c(\mathcal S^+(G, K))$. By the lemma we can choose $f\in C_c(G, K)$ such that $\widehat{f}(\omega) \neq 0$ for all $\omega \in {\rm supp}(\psi)$. Then $\psi/|\widehat{f}|^2$ defines a continuous function on ${\rm supp}(\psi)$, and we can extend this function continuously to all of $\mathcal S^+(G, K)$ by $0$. Now, if $\widehat{\mu}$ is any measure satisfying (God5), then
\begin{equation}\label{LocalFormulaMuHat}
\widehat{\mu}(\psi) = \sigma_f(\psi/|\widehat{f}|^2).
\end{equation}
In particular, there is at most one measure $\widehat{\mu}$ satisfying (God5) or, equivalently, (God1)-(God4). The proof of the existence of a measure satisfying \eqref{LocalFormulaMuHat} is based on the following convolution formula:
 \begin{lemma}\label{GodementConvolution} Let $\varphi \in P(G, K)$ with associated measure $\mu_\varphi$. Then for every $f \in C_c(G, K)$ we have $\varphi \ast f \ast f^* \in P(G,K)$ and its associated measure is given by $|\widehat{f}|^2\mu_\varphi$.
\end{lemma}
\begin{proof} Note first that since $\varphi$ is positive-definite, the measure $\varphi m_G$ is of positive type, and hence $\varphi \ast f \ast f^*  = \varphi m_G \ast f \ast f^* \in P(G,K)$ by Lemma \ref{GodementMainLemma}. Moreover, by Corollary \ref{RadonCommutative} we have for all $g \in G$,
\begin{eqnarray*}
 \varphi \ast f \ast f^*(g) &=& ((f \ast f^*) \ast \varphi)(g) \quad = \quad \int_G (f \ast f^*) (x) \varphi(x^{-1}g) dm_G(x)\\
 &=& \int_{\mathcal S^+(G, K)}\int_G (f\ast f^*)(x) \omega(x^{-1}g)dm_G(x) d\mu_\varphi(\omega)\\
 &=&  \int_{\mathcal S^+(G, K)} ((f \ast f^*) \ast \omega)(g) d\mu_\varphi(\omega).
\end{eqnarray*}
By (S4) and the subsequent remark we have \[((f \ast f^*) \ast \omega)(g) = \mathcal F(f \ast f^*)(\omega) \cdot \omega = |\widehat{f}(\omega)|^2 \cdot \omega.\]
We deduce that
\[
 \varphi \ast f \ast f^*(g)= \int_{\mathcal S^+(G, K)} \omega(g) |\widehat{f}(\omega)|^2 d\mu_\varphi(\omega),
\]
which finishes the proof.
\end{proof}
Now let $f,g \in C_c(G, K)$. Since $C_c(G, K)$ is commutative we have
\[
(\mu \ast f \ast f^*) \ast g \ast g^* = (\mu \ast g \ast g^*) \ast f \ast f^*,
\]
and hence the associated measures must coincide. With Lemma \ref{GodementConvolution} we deduce that
\begin{equation}\label{GodementSymmetry}
|\widehat{g}|^2 \sigma_f = |\widehat{f}|^2 \sigma_g.
\end{equation}
Using this formula one readily concludes:
\begin{lemma}\label{Godement3} There exists a (necessarily unique) positive Radon measure on $\mathcal S^+(G, K)$ which satisfies \eqref{LocalFormulaMuHat} for all $\psi \in C_c(\mathcal S^+(G, K))$ and all $f \in C_c(G,K)$ whose Fourier transform does not vanish on ${\rm supp}(\psi)$. Any such measure satisfies the equivalent Conditions (God1)--(God5).
\end{lemma}
\begin{proof} Let $\psi \in C_c(\mathcal S^+(G, K))$ and let $f,g \in C_c(G)$ such that $\widehat{f}$ and $\widehat{g}$ are positive on ${\rm supp}(\varphi)$; such functions exist by Lemma \ref{Godement2}. It then follows from \eqref{GodementSymmetry} that
$\Lambda(\varphi) :=  \sigma_f(\psi/|\widehat{f}|^2)$ is equal to $\sigma_g(\psi/|\widehat{g}|^2)$, and hence independent of the function $f$ used to define it. Since $\sigma_f$ is a positive continuous linear functional on $C_c(\mathcal S^+(G, K))$ for every $f \in C_c(G, K)$ one concludes that $\Lambda$ is a positive continuous linear functional, hence given by a measure $\widehat{\mu}$. By construction, $\widehat{\mu}$ satisfies (G5).\end{proof}
This completes the proof of Theorem \ref{GodementConvenient}. 

\section{Some estimates concerning Laguerre polynomials}\label{AppendixLaguerre}
In this appendix we collect the estimates concerning Laguerre polynomials which are required for the proof of Lemma \ref{AbsoluteConvergence}.

With the notation of the lemma we set $d := d_1 + d_2$, $V := \C^d$, $\kappa := (\alpha,\beta) \in \bN^d$ and define functions \[
g: V \ra \bC, \quad b_\kappa : V \ra \bC \qand c_\kappa : V \ra [0,\infty)
\]
by the formulas
\[
g = f_{1,\tau_1} \otimes f_{2,\tau_2}, \quad b_{\kappa} = q_{\tau_1,\alpha} \otimes q_{\tau_2,\beta} \qand c_{\kappa} = |b_\kappa|.
\]
Then Lemma \ref{AbsoluteConvergence} amounts to showing that
\begin{equation}
\label{finite}
\sum_{\kappa \in \bN^{d}} \, \big| \langle g, b_{\kappa} \rangle \big|^2 \, \Big( \sum_{\delta \in \Delta} c_{\kappa}(\delta) \Big) < \infty
\end{equation}
To show this, it suffices to establish the following two Properties (P1) and (P2):
\begin{enumerate}
\item[(\textsc{P1})] For every $M \geq 1$, 
\[
\big| \langle g, b_{\kappa} \rangle \big| \ll_M \big(\kappa_1 \cdots \kappa_d)^{-M}, \quad \textrm{for all $\kappa \in \bN^d$}.
\]
\vspace{0.2cm}
\item[(\textsc{P2})] There exists $M_o \geq 1$ such that
\[
\sum_{\delta \in \Delta} c_{\kappa}(\delta) \ll (\kappa_1 \cdots \kappa_d)^{M_o} \quad \textrm{for all $\kappa \in \bN^d$}.
\]
\end{enumerate}
To establish (P1), we recall that
\[
b_\kappa(v_1,\ldots,v_d) = \prod_{j=1}^d \Big( e^{-\eta_{j} |v_j|^2/4} L_{\alpha_j}(\eta_{j} |v_j|^2/2) \Big)
\qand
c_\kappa = |b_\kappa|,
\]
for some $\eta = (\eta_1,\ldots,\eta_d) \in \bR^{d}_{+}$. Let us first pretend that $g$ is of the form
\[
g(v) = g_1(|v_1|^2) \cdots g_{d}(|v_d|^2), \quad \textrm{for $v = (v_1,\ldots,v_d) \in \bC^d$},
\]
since the general case is not much harder except for notation. After a straightforward variable substitution, we find a positive number $A_\eta$ and $\xi = (\xi_1,\ldots,\xi_d) \in \bR_{+}^r$ such that
\[
\langle g, b_\kappa \rangle = A_{\eta} \prod_{j=1}^d \Big( \int_0^\infty g_j(\xi_{j} t_j) L_{\kappa_j}(t) e^{-t/2} \, dt \Big).
\]
Let us write $h_j(t) = e^{t/2} g_j(\xi_{j} t_j)$. To prove \textsc{(P1)} it suffices to show that for all $M \geq 1$,
\[
\Big| \int_0^\infty h_j(t) L_{n}(t) e^{-t} \, dt \Big|  \ll_M n^{-M}, \quad \textrm{for all $n$}.
\]
Since $h_j$ is compactly supported and smooth, \cite[Theorem 2.1]{S} tells us that 
\[
\Big| \int_0^\infty h_j(t) L_{n}(t) e^{-t} \, dt \Big| \ll_M n^{-M} \Big( \int_0^\infty t^{1+r} |h_j^{(2M)}(t)|^2 e^{-t} \, dt \Big)^{1/2}
\ll_M n^{-M},
\]
for all $M$, and thus \textsc{(P1)} is established for $g$ in the above product form. If $g$ is not such a product, then one applies the same argument inductively, freezing all but one variables at a time.

We now turn to the proof of (P2). By Proposition \ref{tauProperties}.(ii) we have $b_{\kappa} = b_{\kappa} \ast b_{\kappa}$, hence taking absolute values yields
\begin{equation}
\label{superidempotent}
c_{\kappa} \leq c_{\kappa} * c_{\kappa}, \quad \textrm{for all $\kappa$}.
\end{equation}
\begin{lemma} If $c : V \ra [0,\infty)$ is a function satisfying $c\leq c * c$, then for every sub-multiplicative function $\rho : V \ra (0,\infty)$  we have
\[
c(v) \leq \rho(v)^{-1/2} \int_V c(u)^2 \, \rho(u) \, du, \quad \textrm{for all $v \in V$}.
\]
\end{lemma}

\begin{proof} For every $v \in V$ we have the estimate
\begin{eqnarray*}
c(v) 
&\leq &
\int_V c(u) c(v-u) \, du = \int_V c(u) \rho(u)^{1/2} \, c(v-u) \rho(v-u)^{1/2} \, \big(\rho(u) \rho(v-u)\big)^{-1/2} \, du \\[0.2cm]
&\leq &
\rho(v)^{-1/2} \, \int_V c(u) \rho(u)^{1/2} \, c(v-u) \rho(v-u)^{1/2} \, du \leq \rho(v)^{-1/2} \, \int_V c(u)^2 \rho(u) \, du,
\end{eqnarray*}
where the last inequality holds by Cauchy-Schwarz.
\end{proof}
In view of \eqref{superidempotent} we thus have
\begin{equation}
\sum_{\delta \in \Delta} c_{\kappa}(\delta) \leq \Big( \sum_{\delta \in \Delta} \rho(\delta)^{-1/2} \Big) \, \int_V c_\kappa(v)^2 \, \rho(v) \, dv.
\end{equation}
for every submultiplicative function $\rho : V \ra (0,\infty)$. To establish (\textsc{P2}), it is then enough to find a sub-multiplicative $\rho$ such that the following hold:
\vspace{0.2cm}
\begin{enumerate}
\item[\textsc{(P3)}] $\sum_{\delta \in \Delta} \rho(\delta)^{-1/2} < \infty$.
\vspace{0.2cm}
\item[\textsc{(P4)}] The map 
\[
\kappa \mapsto \int_V c_{\kappa}(v)^2 \, \rho(v) \, dv
\]
grows at most polynomially. 
\end{enumerate}
The sub-multiplicative functions that we will use will be of the form 
\[
\rho_N(v) = (1+\|v\|)^N, \quad \textrm{where } v \in V \text{ and }\|v\|^2 = |v_1|^2 + \ldots + |v_d|^2.\]
If $N$ is large enough, (\textsc{P3}) is clearly satisfied, and to establish (\textsc{P4})
we only need to show that for every $r$, the map
\begin{equation}
\label{polynomial}
\kappa \mapsto \int_V c_{\kappa}(v)^2 \, \|v\|^r \, dv
\end{equation}
grows at most polynomially (with a degree which is allowed to depend on $r$). Upon expanding the norm $\|\cdot\|$ and using the product structure of $c_\kappa$, condition \eqref{polynomial} amounts to proving
that for every integer $r$, the map
\[
n \mapsto \int_0^\infty t^r L_n(t)^2 \, e^{-t} \, dt
\]
grows at most polynomially. We recall that
\[
\int_0^\infty L_n(t) L_m(t) e^{-t} \, dt = \delta_{mn} 
\]
and 
\[
t L_n(t) = (2n+1)L_n(t) - n L_{n-1}(t) - (n+1)L_{n-1}(t), \quad \textrm{for all $n$}.
\]
Hence, if $r$ is an integer, $t^r L_n$ is a linear combination of the Laguerre polynomials $L_{n+j}$ for $|j| \leq r$, with coefficients which are polynomials
in $n$ of degrees at most $r$. If we denote by $\beta_n$ the coefficient in front of $L_n$, we conclude that
\[
\int_0^\infty t^r L_n(t)^2 \, e^{-t} \, dt = \beta_n,
\]
which is a polynomial of $n$ of degree at most $r$. This proves \textsc{\ref{polynomial}}, whence \textsc{(P4)}, and we are done.

\bibliographystyle{abbrv}

\end{document}